\documentclass[11pt,reqno]{amsart}
\usepackage{enumitem}

\numberwithin{equation}{section}

\newtheorem{claim}{Claim}[section]
\newtheorem{theorem}{Theorem}[section]

\newtheorem{proposition}[theorem]{Proposition}
\newtheorem{lemma}[theorem]{Lemma}
\newtheorem{corollary}[theorem]{Corollary}
\newtheorem{remark}[theorem]{Remark}

\def\al{\aligned}
\def\eal{\endaligned}
\def\be{\begin{equation}}
\def\ee{\end{equation}}
\def\lab{\label}
\def\a{\alpha}

\def\e{\epsilon}

\def\M{{\bf M}}

\def\al{\aligned}

\def\p{\partial}
\def\d{\nabla}

\def\k{\kappa}
\def\v{\textrm{Vol}}

\usepackage{soul,color}
\numberwithin{equation}{section}

\usepackage[bindingoffset=0.2in,%
            left=1in,right=1in,top=1in,bottom=1in,%
            footskip=.25in]{geometry}
\usepackage{esint, cancel, ulem, thmbox, mdframed}

\begin{document}

\title[Li-Yau gradient bounds under nearly optimal curvature conditions]{Li-Yau gradient bounds on compact manifolds under nearly optimal curvature conditions}
\author{Qi S. Zhang and Meng Zhu}
\address{Q. S. Z.: Department of Mathematics, University of California, Riverside, Riverside, CA 92521, USA}
\email{qizhang@math.ucr.edu}
\address{M. Z.: Department of Mathematics, Shanghai Key Laboratory of Pure Mathematics and Mathematical Practice, East China Normal University, Shanghai 200241, China}
\address{Department of Mathematics, University of California, Riverside, Riverside, CA 92521, USA}
\email{mzhu@math.ecnu.edu.cn}

\begin{abstract}
We prove Li-Yau type gradient bounds for the heat equation either on manifolds with fixed metric or
under the Ricci flow. In the former case the curvature condition is $|Ric^-| \in L^p$ for some
$p>n/2$, or $\sup_\M \int_\M |Ric^-|^2(y)d^{2-n}(x,y)dy<\infty$,  where $n$ is the dimension of the manifold. In the later case, one only needs scalar curvature being bounded. We will explain why the conditions are nearly optimal and give an application. The Li-Yau bound for the heat equation on manifolds with fixed metric seems to be the first one allowing Ricci curvature not bounded from below.
\end{abstract}
\maketitle
\tableofcontents

\section{Introduction}

Let $(\M^n, g_{ij})$ be an $n$-dimensional complete Riemannian manifold. In \cite{LY:1}, P. Li an S.T. Yau discovered the following celebrated Li-Yau bound, for positive solutions of the heat equation
\be
\lab{heatequation}
\frac{\partial u}{\partial t}=\Delta u.
\ee Suppose  $Ric\geq-K$, where $K\geq0$ and $Ric$ is the Ricci curvature of $\M$. Then any positive solution of \eqref{heatequation} satisfies
\begin{equation}\label{Li-Yau>1}
\frac{|\nabla u|^2}{u^2} - \alpha\frac{u_t}{u} \leq \frac{n\alpha^2K}{2(\alpha-1)}+\frac{n\alpha^2}{2t},\quad \forall \alpha>1.
\end{equation}
In the special case where $Ric\geq 0$, one has the optimal Li-Yau bound
\begin{equation}\label{Li-Yau}
\frac{|\nabla u|^2}{u^2} - \frac{u_t}{u} \leq \frac{n}{2t}.
\end{equation}
In the same paper, many applications of \eqref{Li-Yau>1} and \eqref{Li-Yau} have also been demonstrated by the authors, including the classical parabolic Harnack inequality, optimal Gaussian estimates of the heat kernel, estimates of eigenvalues of the Laplace operator, and estimates of the Green's function. Moreover, \eqref{Li-Yau>1} and
\eqref{Li-Yau} can even imply the Laplacian Comparison Theorem (see e.g. \cite{Chowetc} page 394).

The Li-Yau bound \eqref{Li-Yau>1} was later improved for small time by Hamilton in \cite{Ha:2}, where he proved under the same assumptions as above that
\begin{equation}\label{Hamilton}
\frac{|\nabla u|^2}{u^2} - e^{2Kt}\frac{u_t}{u} \leq e^{4Kt}\frac{n}{2t}.
\end{equation}
Hamilton \cite{Ha:2} further showed a matrix Li-Yau bound for the heat equation. Similar matrix Li-Yau bound was subsequently obtained by Cao-Ni \cite{CaNi} on K\"ahler manifolds.

For the past three decades, many Li-Yau type bounds have been proved not only for
the heat equation, but more generally, for other linear and semi-linear parabolic equations
on manifolds with or without weights.  Let us mention the result by  Bakry and Ledoux \cite{BL} who
derived the Li-Yau bound for weighted manifolds by an ordinary differential inequality
involving the entropy and energy of the backward heat equation. For most recent development, see the papers \cite{CTZ}, \cite{Dav}, \cite{GM},
\cite{LX}, \cite{QZZ}, \cite{Wan}, \cite{WanJ} and the latest \cite{BBG}, \cite{ZZ} and references therein. In all of these results, the essential assumption is that   the Ricci curvature or the corresponding Bakry-Emery Ricci curvature is bounded from below by a constant.
In many situations, it is highly desirable to weaken this assumption.

 Li-Yau bounds have also been extended to situations with moving metrics.
Let $g_{ij}(t)$, $t\in[0, T]$, be a family of Riemannian metrics on $\M$ which solves the Ricci flow:
\begin{equation}\label{RF1}
\frac{\partial }{\partial t}g_{ij}(t)=-2R_{ij}(t),
\end{equation}
where $R_{ij}(t)$ is the Ricci curvature tensor of $g_{ij}(t)$. One may still consider linear and semi-linear parabolic equations under the Ricci flow in the sense that in the heat operator $\frac{\partial}{\partial t}-\Delta$, we have $\Delta=\Delta_t$ which is the Laplace operator with respect to the metric $g_{ij}(t)$ at time $t$. The two most prominent examples are the heat equation
\be
\label{ricciheateq}
  (\Delta-\frac{\partial }{\partial t})u= 0, \, \partial_t g_{ij}= - 2 R_{ij}
  \ee and the conjugate heat equation
 \be
\label{ricciconheateq}
  (\Delta-R + \frac{\partial }{\partial t})u= 0, \, \partial_t g_{ij}= - 2 R_{ij}.
  \ee

The study of Li-Yau bound for heat type equations under the Ricci flow was initiated by Hamilton. In \cite{Ham1988}, he obtained a Li-Yau bound for the scalar curvature along the Ricci flow on 2-sphere. This result was later improved by Chow \cite{Chow}. In higher dimensions, both matrix and trace Li-Yau bounds for curvature tensors, also known as Li-Yau-Hamilton inequalities, were obtained by Hamilton \cite{Ham1993jdg} for the Ricci flow with bounded curvature and nonnegative curvature operator. These estimates played a crucial role in the study of singularity formations of the Ricci flow on three-manifolds and solution to the Poincar\'e conjecture. We remark that
Brendle \cite{Bre} has generalized Li-Yau-Hamilton inequalities under weaker curvature assumptions. The Li-Yau-Hamilton inequality for the K\"ahler-Ricci flow with nonnegative holomorphic bisectional curvature was obtained by H.-D. Cao \cite{Cao}. In addition, in \cite{P:1}, Perelman showed a Li-Yau type bound for the fundamental solution of the conjugate heat equation \eqref{ricciconheateq} under the Ricci flow (see also \cite{Ni2006}).

Recently, there have been many results on Li-Yau bounds for positive solutions of the heat or conjugate heat equations under the Ricci flow. For example authors of \cite{KuZh} and \cite{Cx} proved Li-Yau type bound
for all positive solutions of the conjugate heat equation without any curvature condition, just like
Perelman's aforementioned result for the fundamental solution.
In \cite{CH:1} and \cite{BCP} the authors  proved various Li-Yau type bounds for positive solutions of \eqref{ricciheateq}
under either positivity condition of the curvature tensor or boundedness of the Ricci curvature.  So there is a marked difference between these results on the conjugate
heat equation and the heat equation in the curvature conditions. In view of the absence of curvature
condition for the conjugate heat equation, one would hope that the curvature conditions for the heat equation can be weakened.

Recently, in \cite{BZ}, the authors proved the following gradient estimate for bounded positive solutions $u$ of the heat equation \eqref{ricciheateq},
\begin{equation}\label{BZ}
|\Delta u|+\frac{|\nabla u|^2}{u}-aR\leq \frac{Ba}{t},
\end{equation}
where $R=R(x,t)$ is the scalar curvature of the manifold at time $t$, and $B$ is a constant and $a$ is an upper bound of $u$ on $M\times[0,T]$. Although this result requires no curvature condition and it has some other applications,  it is not a
Li-Yau type bound.

The goal of this paper is to prove Li-Yau bounds for positive solutions for both the fixed metric case  \eqref{heatequation} and  the Ricci flow case
\eqref{ricciheateq} under essentially optimal curvature conditions.

The first theorem is for the fixed metric case, we will have two independent conditions
and two conclusions. The conditions are motivated by different problems such as studying
manifolds with integral Ricci curvature bound and the K\"ahler-Ricci flow. The conclusions
range from long time bound with necessarily worse constants, to short time
bound with better constants.

\begin{theorem}\label{thmLY}
Let $(\M, g_{ij})$ be a compact $n$-dimensional Riemannian manifold, and $u$ a positive solution of \eqref{heatequation}.  Suppose either one of the following conditions holds.

(a)  $\int_\M |Ric^-|^p dy \equiv \sigma <\infty$ for some $p>\frac{n}{2}$, where $Ric^-$ denotes the nonpositive part of the Ricci curvature;  and the manifold is noncollapsed under scale 1, i.e., $|B(x, r)| \ge \rho r^n$ for $0<r\leq 1$ and some
 $\rho>0$;

(b)   $\sup_\M \int_\M |Ric^-|^2 d^{-(n-2)}(x,y)dy \equiv \sigma <\infty$ and the heat kernel of \eqref{heatequation}
satisfies the Gaussian upper bound (which holds automatically under (a)):
\be
\lab{GUP}
G(x, t; y, 0) \le \frac{\hat C(t)}{t^{n/2}} e^{ - \bar c d^2(x, y)/t},  \, t \in (0, \infty)
\ee for some positive constant $\bar c$ and positive increasing function $\hat C(t)$ which grows to infinity as $t\to \infty$.  Here $d(x,y)$ is the distance from $x$ to $y$.\\

\hspace{-.5cm}Then,

(1) for any constant $\alpha\in(0,1)$, we have
\be\label{eqLY}
\a \underline{J}(t) \frac{ |\d u|^2}{u^2} - \frac{\p_t u}{u} \le \frac{n}{(2 - \delta)\a\underline{J}(t)} \frac{1}{t}
\ee
for $t\in (0, \infty)$, where $$\delta=\frac{2(1-\a)^2}{n+(1-\a)^2},$$ and
\be
\underline{J}(t)=\left\{
\al
&2^{-1/(5\delta^{-1}-1)}e^{-(5\delta^{-1}-1)^{\frac{n}{2p-n}}\left[4\sigma\hat C(t)^{1/p}\right]^{\frac{2p}{2p-n}}t},\ \textrm{under condition (a)};\\
&2^{-1/(10\delta^{-1}-2)}e^{-C_0(5\delta^{-1}-1)\sigma \hat C(t)t},\ \textrm{under condition (b)},
\eal\right.
\ee
with $C_0$ being a constant depending only on $n$, $p$ and $\rho$, and $\hat C(t)$ the increasing function on the right hand side of \eqref{GUP}.

(2) in particular, for any $\beta\in(0,1)$, there is a $T_0=T_0(\beta,\sigma, p, n, \rho)$ such that
\be\label{eqLY2}
\beta\frac{ |\d u|^2}{u^2} - \frac{\p_t u}{u} \le \frac{n}{2\beta} \frac{1}{t}
\ee
for $t\in (0, T_0]$. Here $T_0 = c (1-\beta)^{4p/(2p-n)}$ and $c (1-\beta)^{4}$
under conditions (a) and (b), respectively; and $c$ is a positive constant depending only
on the parameters of conditions (a) and (b), i.e., $c=c(\sigma, p, n, \rho) $.

\end{theorem}

\begin{remark}

It is well known that Condition (a) actually implies that the heat kernel of \eqref{heatequation} has a Gaussian upper bound for all time, i.e. \eqref{GUP} holds:
\[
G(x, t; y, 0) \le \frac{\hat C(t)}{t^{n/2}} e^{ - \bar c d^2(x, y)/t},  \, t \in (0, \infty)
\] for some positive constant $\bar c$ and positive increasing function $\hat C(t)$ which grows to infinity as $t\to \infty$. For short time interval $(0,1]$, the function $\hat C(t)$ can be replaced by a constant $\hat C$. This is proven in \cite{TZz:1} Section 2 e.g.. Longer time bound follows from the reproducing formula of heat kernels.
In addition,  a volume upper bound $|B(x, r)| \le C r^n$ follows from Petersen-Wei \cite{PW1}.

The condition $|Ric^-|\in L^p$ for some $p>\frac{n}{2}$ is nearly optimal in the sense that it is not clear $|\nabla u|$ will stay bounded when $|Ric^-|\in L^{\frac{n}{2}}$, which is the well known border
line condition where regularity may fail.

If $Ricci \ge 0$,  one can take $\sigma=0$ and $\alpha \to 1$. Then \eqref{eqLY} becomes the optimal
Li-Yau bound \eqref{Li-Yau}.

Professor Guofang Wei kindly informed us that the noncollapsing condition in (a) may possibly  be removed
by a recent result of Dai-Wei-Z. L. Zhang. In this case the Gaussian upper bound \eqref{GUP} should be replaced by a local version where $t^{n/2}$ is replaced by $|B(x, \sqrt{t})|$ with suitable adjusted parameters.

\end{remark}

\begin{remark}
The conditions in (b) imply that $|Ric^-|^2$ belongs to the Kato class, i.e.,
\[
K(|Ric^-|^2)=\sup_{x\in\M}\int_M \Gamma(x,y)|Ric^-|^2(y)dy<\infty,
\]
where $\Gamma(x,y)$ denotes the Green's function on $\M$.

However, it is not hard to see that the same techniques used in the proof can actually be applied to deal with the cases where $|Ric^-|^p$ is in the Kato class for any $1\leq p<\infty$. Note $p=1$ is the optimal one scaling wise.

The main reason that we are considering the curvature condition in (b) is that it is preserved under the K\"ahler-Ricci flow as proved in \cite{TZq1} and \cite{TZq2}, while it is not clear whether the same property holds if one replaces the power of $|Ric^-|$ by some $p>2$. It is also proven in \cite{TZq1} that the Gaussian bound for the heat kernel holds for each time slice of the normalized K\"ahler Ricci flow with uniform constants. Therefore, the conclusion of part (b) is valid with uniform constants for each time slice along the normalized K\"ahler Ricci flow.
\end{remark}

The Li-Yau bound in the above theorem seems to be the first one allowing Ricci curvature not bounded from below.
Moreover, as an application, we use \eqref{eqLY2} to extend some results in \cite{CoNa} on the parabolic approximations of the distance functions to the case where $|Ric^-|\in L^p$ for some $p>\frac{n}{2}$. The main extension results were first proved in \cite{TZz:1}.\\

Next we turn to the heat equation coupled with the Ricci flow \eqref{ricciheateq} for which we prove
\begin{theorem}\label{thmLYRF}
Let $\M$ be a compact $n$-dimensional Riemmannian manifold, and $g_{ij}(t)$, $t\in [0,T)$, a solution of the Ricci flow \eqref{RF1} on $\M$. Denote by $R$ the scalar curvature of $\M$ at $t$, and $R_1$ a positive constant. Suppose that $-1\leq R\leq R_1$ for all time $t$, and $u$ is a positive solution of the heat equation \eqref{ricciheateq}.
Then, for any $\delta\in[\frac{1}{2},1)$, we have
\begin{equation}\label{LYRF}
\delta\frac{|\nabla u|^2}{u^2}-\frac{\partial_t u}{u} \le \delta\frac{|\nabla u|^2}{u^2}-\frac{\partial_t u}{u}-\alpha R+\frac{\beta}{R+2}\leq
\frac{1}{t} \left(\frac{n}{2\delta}+\frac{4n\beta T}{\delta(1-\delta)} \right)
\end{equation}
for $t\in(0,T)$, where $\alpha=\frac{n}{2\delta(1-\delta)^2}$ and $\beta=\alpha (R_1+2)^2$.

\end{theorem}

\begin{remark}
Note that the curvature assumption is made only on the scalar curvature rather than on the Ricci or curvature tensor. For the Ricci flow on a compact manifold $\M$, one can always rescale a solution so that the scalar curvature to be bounded from below by $-1$. Under suitable assumptions, the
result in the theorem still holds when $\M$ is complete noncompact.

The Li-Yau bound in the above theorem actually is scaling invariant. Readers can refer to Theorem \ref{thmLYRF1} in section 3 for the corresponding version before rescaling the metrics.
In case the scalar curvature is 0, the Ricci curvature is also $0$ by the maximum principle. Then, by scaling,
we can let $\delta=1$ in the theorem and the bound becomes the optimal Li-Yau bound for Ricci flat case.

\end{remark}

\begin{remark} This theorem clearly implies a Harnack inequality for positive solutions of \eqref{ricciheateq}
if the scalar curvature is bounded.
\end{remark}

This paper is organized as follows: the main Theorems \ref{thmLY} and \ref{thmLYRF} are proved in sections 2 and 3, respectively. The main technical hurdle is to construct certain auxiliary functions to
cancel various curvature terms arising from commutation formulas. For example, in order to
prove Theorem \ref{thmLY}, one needs to deal with the bad term $| Ric^-| \frac{|\nabla u|^2}{u}$.
If one only imposes integral conditions on $|Ric^-|$, then this term can not be bounded by good terms coming out of the Bochner's formula.  The auxiliary functions for Theorem \ref{thmLY}  are obtained by solving
a nonlinear evolution equation, which is used to cancel the bad term.
When proving  Theorem \ref{thmLYRF} in section 3, an additional bad term $< \nabla^2 u, Ric >$
appears.
We will use the good terms coming from the equation of $\frac{\beta}{R}$ to
control it.
In section 4, we deduce from Theorem \ref{thmLY} the extended parabolic approximations of the distance functions.
\medskip

\section{Fixed metric case}

In this section, we work on an $n$-dimensional compact Riemannian manifold $\M$ with a fixed
metric $g$.   For the Ricci curvature, we assume either
\be
\lab{ric-cond}
| Ric^- | \in L^p(\M), \quad p>n/2, \quad \text{or} \quad \sup_{x \in \M}
\int_{\M} \frac{ |Ric^-(y)|^2}{d^{n-2}(x, y)} dy < \infty.
\ee

\hspace{-.5cm}{\bf Proof of Theorem \ref{thmLY}:}
By direct computation, we have
\[
(\Delta  - \p_t) \frac{ |\d u|^2}{u} = \frac{2}{u} \left| u_{ij} - \frac{u_i u_j}{u}
\right|^2 + 2 R_{ij} \frac{u_i u_j}{u}.
\]

Let $J=J(x, t)$ be a smooth positive  function and $\a \in (0, 1)$ be a parameter.   Then
\[
\al
&(\Delta  - \p_t)  \left[ \a J \frac{ |\d u|^2}{u} - \p_t u \right] \\
&= \a \left[ \frac{2}{u} \left| u_{ij} - \frac{u_i u_j}{u}
\right|^2 J + 2 R_{ij} \frac{u_i u_j}{u} J + \Delta J \frac{ |\d u|^2}{u}
+ 2 \d J \cdot \d \frac{ |\d u|^2}{u}  -(\p_t J)  \frac{ |\d u|^2}{u} \right].
\eal
\]  Denote the heat operator $\Delta-\frac{\partial}{\partial t}$ by $\mathcal{L}$. Recall the quotient formula for the heat operator.
\[
\mathcal{L} \left(\frac{F}{G}\right) + 2 \d \ln G \cdot \d \frac{F}{G} =
\frac{\mathcal{L}F}{G} - \frac{F \mathcal{L}G}{G^2}.
\]Take
\[
F= \a J \frac{ |\d u|^2}{u} - \p_t u, \quad G=u,
\] and
\be
\lab{defforQ}
Q \equiv \a J \frac{ |\d u|^2}{u^2} - \frac{\p_t u}{u}.
\ee
We find that
\be
\lab{eqforQ}
\al
&(\Delta  - \p_t) Q
+ 2 \frac{\d u}{u}\cdot \d Q \\
&= \a \left[ \frac{2}{u^2} \left| u_{ij} - \frac{u_i u_j}{u}
\right|^2 J + 2 R_{ij} \frac{u_i u_j}{u^2} J + \Delta J \frac{ |\d u|^2}{u^2}
+
2 \d J\cdot \d \left(\frac{ |\d u|^2}{u} \right) \frac{1}{u}
 -(\p_t J)  \frac{ |\d u|^2}{u^2} \right].
\eal
\ee

Let $f=\ln u$. Using the identities
\[
\frac{1}{u^2} \left| u_{ij}  - \frac{u_i u_j}{u}
\right|^2  = | f_{ij}|^2,
\]
and
\[
\d \left(\frac{ |\d u|^2}{u} \right) \frac{1}{u}
= \d \left(\frac{ |\d u|^2}{u^2} \right) + \frac{ |\d u|^2}{u} \frac{\d u}{u^2}
= \d \left( |\d f|^2 \right)+ \frac{ |\d u|^2}{u} \frac{\d u}{u^2},
\]we can turn (\ref{eqforQ}) into
\be
\lab{eqforQ2}
\al
&(\Delta  - \p_t) Q
+ 2 \frac{\d u}{u}\cdot \d Q \\
&= \a \left[ 2 \left| f_{ij}\right|^2 J + 2 R_{ij} \frac{u_i u_j}{u^2} J
+ \Delta J \frac{ |\d u|^2}{u^2}
+
2 \d J \cdot\d |\d f|^2 + 2 \d J \cdot\frac{\d u}{u} \left(\frac{ |\d u|^2}{u^2} \right)
 -(\p_t J)  \frac{ |\d u|^2}{u^2} \right].
\eal
\ee Observe that, in local coordinates,
\[
2 \d J\cdot \d |\d f|^2 =2 J_i (f^2_j)_i = 4 J_i f_{ji} f_j
\ge - \delta | f_{ij} |^2 J - \frac{| \d J|^2}{J} \, |\d f|^2 4 \delta^{-1}.
\]Therefore, we deduce the following inequality
\be
\lab{ineqforQ}
\al
&(\Delta  - \p_t) Q
+ 2 \frac{\d u}{u}\cdot \d Q \\
&\ge  \a \left[  (2 J - \delta J) \left| f_{ij}\right|^2  - 2 |Ric^-| \frac{|\d u|^2}{u^2} J
+ \Delta J \frac{ |\d u|^2}{u^2}
-\frac{| \d J|^2}{J} \, |\d f|^2 4 \delta^{-1}\right. \\
 &\qquad  \left. -2 |\d J|   \frac{ |\d u|^3}{u^3}
 -(\p_t J)  \frac{ |\d u|^2}{u^2} \right].
\eal
\ee Using the inequality, for any $\delta>0$,
\[
2 |\d J|   \frac{ |\d u|^3}{u^3} =
2 |\d J | \, |\d f|^3  \le \delta J |\d f|^4 + \delta^{-1} \frac{| \d J|^2}{J} \, |\d f|^2,
\]we can turn the above inequality into
\[
\al
&(\Delta  - \p_t) Q
+ 2 \frac{\d u}{u}\cdot \d Q \\
&\ge  \a  (2 J - \delta J) \left| f_{ij}\right|^2  +
\a \left[\Delta J - 2 |Ric^-|  J
-
5\delta^{-1} \frac{| \d J|^2}{J}
 -\p_t J \right] |\d f|^2 - \delta \a J |\d f|^4.
\eal
\] From \eqref{heatequation}, we know
$$\Delta f-\partial_t f=-|\nabla f|^2.$$ Hence,
\be
\lab{ineqtQ}
\al
&(\Delta  - \p_t) ( t Q )
+ 2 \frac{\d u}{u} \cdot\d (t Q)  \\
&\ge  \a t (2 J - \delta J) \frac{1}{n}  \left( | \d  f |^2 - \p_t f \right)^2  +
\a \left[\Delta J - 2 V  J
-
5 \delta^{-1} \frac{| \d J|^2}{J}
 -\p_t J \right]  t |\d f|^2\\
&\qquad  - \delta \a  t J |\d f|^4 -Q,
\eal
\ee where we have written $| Ric^-|=V$.

For any given parameter $\delta>0$ such that $5\delta^{-1}>1$, we make the following
\begin{claim}
\lab{claim}
the problem
\be
\lab{eqforJ}
\begin{cases}
\Delta J - 2 V  J
-
5 \delta^{-1}\frac{| \d J|^2}{J}
 -\p_t J =0, \quad \text{on} \quad {\M} \times (0, \infty);\\
J(\cdot, 0) = 1,
\end{cases}
\ee
has a unique solution for $t\in[0,\infty)$, which satisfies
\be
\underline{J}(t)\leq J(x,t)\leq 1,
\ee
where
\be
\underline{J}(t)=\left\{
\al
&2^{-1/(5\delta^{-1}-1)}e^{-(5\delta^{-1}-1)^{\frac{n}{2p-n}}\left[4\sigma\hat C(t)^{1/p}\right]^{\frac{2p}{2p-n}}t},\ \textrm{under condition (a)};\\
&2^{-1/(10\delta^{-1}-2)}e^{-C_0(5\delta^{-1}-1)\sigma \hat C(t)t},\ \textrm{under condition (b)},
\eal\right.
\ee
with $C_0$ being a constant depending only on $n$, $p$ and $\rho$, and $\hat C(t)$ the increasing function in \eqref{GUP}.

\end{claim}

In the following steps, we will prove the claim.\\

\hspace{-.45cm}{\it step 1.   Conversion into an integral equation.}

Let $a=5\delta^{-1}$, and
\be
\lab{defw}
w = J^{-(a-1)}.
\ee
It is straightforward to check that $w$ satisfies
\be
\lab{eqforw}
\begin{cases}
\Delta w -\p_t w +2(a-1) V w =0 , \quad \text{on} \quad {\M} \times (0, \infty);\\
w(\cdot, 0) =1.
\end{cases}
\ee
Since $V$ is a piece-wise smooth function, \eqref{eqforw} has a long time solution (see e.g. Chapter 6 of \cite{Lieb}).

To show that $J$ exists for all time and derive the bounds for $J$, we derive the bounds for $w$ first. Via the Duhamel's formula, \eqref{eqforw} can be transformed to the following
integral equation,
\be
\lab{inteqforw}
w(x, t) = 1+ 2(a-1)\int^t_0 \int_\M G(x, t; y, s) V(y) w (y, s) dyds.
\ee Here $G(x,t;y,s)=G(y,t;x,s)$ is the heat kernel on $\M$.\\

\hspace{-.45cm}{\it step 2. long time bounds}

Here we prove long time bounds for solutions of \eqref{eqforw}.

Let $w$ be a solution of \eqref{eqforw}. For a lower bound of $w$, we can show that
\be
\lab{w>w0}
w  \ge 1
\ee for any $t>0$.

In fact, let $\e>0$ be a small positive number, which will be taken to $0$ eventually. Then the
function $Z_\e = e^{\e t} w$ satisfies the equation
\be
\lab{eqforwe}
\Delta Z_\e + 2(a-1) V Z_\e
 -\p_t Z_\e = - \e Z_\e.
\ee
First, by continuity, since $w(\cdot, 0)=1$, we know that $w\geq0 $ at least for a short time. Applying the maximum principle on \eqref{eqforw}, we see that \eqref{w>w0} holds
at least for a short time. So $Z_\e > 1$ at least for a short time. We now show that
\be
\lab{ee>w0}
Z_\e > 1
\ee for all time $t>0$ as long as the
solution exists.  Suppose not.  Then there exists a first time $t_0$ and  point $x_0 \in \M$ such that
$Z_\e(x_0, t_0) = 1$. At this point $(x_0, t_0)$,  the following holds
\[
\Delta Z_\e \ge 0, \quad \partial_t Z_\e \le 0, \quad 2(a-1) V Z_\e  \ge 0.
\]This is a contradiction to equation \eqref{eqforwe}.  Letting $\e \to 0$ in \eqref{ee>w0},  we know \eqref{w>w0} holds for all time.

Notice that \eqref{w>w0} implies that $J\leq 1$ as long as the solution exits, which is not obvious to see from \eqref{eqforJ}.

Next, for any fixed $T>0$, we derive an upper bound for $w(x,t)$ on $[0,T]$. We will treat condition (a) and (b) separately.

Let $$m(t)=\sup_{\M\times[0,t]}w(x,s).$$
First, under condition (b), since
\[
\al
w(x, t) &= 1+2(a-1)\int_0^t\int_\M G(x,t;y,s)V(y)w(y,s)dyds\\
&\le 1 +2(a-1)  \left(\int^t_0 \int_\M G(x, t; y, s) V^2(y)dyds\right)^{1/2} \left( \int^t_0 \int_\M G(x, t; y, s) w^2(y,s)dyds\right)^{1/2}\\
& \le  1+  2(a-1)  \left(\int^t_0 \int_\M \frac{\hat C(t)}{(t-s)^{n/2}}e^{-\frac{\bar{c} d^2(x,y)}{t-s}}V^2(y)dyds\right)^{1/2}\left( \int^t_0 \int_\M G(x, t; y, s) m^2(s)dyds\right)^{1/2}\\
&\leq 1+C_0(a-1)\sqrt{\hat C(t)}\left( \int_\M \frac{V^2(y)}{d^{n-2}(x,y)}dy\right)^{1/2}\left(\int^t_0m^2(s)ds\right)^{1/2}\\
&\leq 1+C_0(a-1)\sqrt{\sigma\hat C(T)}\left(\int^t_0m^2(s)ds\right)^{1/2}
\eal
\]
for all $t\in[0,T]$ and $x\in\M$, we have
$$
m(t)\leq 1+C_0(a-1)\sqrt{\sigma\hat C(T)}\left(\int^t_0m^2(s)ds\right)^{1/2},
$$
and hence
$$
m^2(t)\leq 2+2C_0(a-1)^2\sigma\hat C(T)\int^t_0m^2(s)ds,
$$
which is the Gr\"onwall inequality.

Therefore, we get
\begin{align}
m^2(t)\leq 2e^{2C_0(a-1)^2\sigma\hat C(T)t},
\end{align}
i.e.,
\begin{align}
m(t)\leq \sqrt{2}e^{C_0(a-1)^2\sigma\hat C(T)t}.
\end{align}
Especially, we have shown
\begin{align}
w(x,t)\leq \sqrt{2}e^{C_0(a-1)^2\sigma\hat C(t)t},
\end{align}
for any $t\in[0,\infty)$.
From \eqref{defw}, we have
$$ 2^{-1/(2a-2)}e^{-C_0(a-1)\sigma \hat C(t)t}\leq J(x,t)\leq 1.$$

Under the condition (a),
\[
\al
w(x, t) &= 1+2(a-1)\int_0^t\int_\M G(x,t;y,s)V(y)w(y,s)dyds\\
&\le 1 +2(a-1)\int^t_0 \int_\M G(x, t; y, s) V(y) m(s)dyds.
\eal
\]
Thus,
\be\label{m(t)1}
\al
m(t)&\le 1 +2(a-1)\int^t_0 \int_\M G(x, t; y, s) V(y) m(s)dyds\\
&=  1+ 2(a-1)\int^{t-\e}_0 \int_\M G(x, t; y, s) V(y) m(s)dyds\\
&\quad + 2(a-1)\int^t_{t-\e} \int_\M G(x, t; y, s) V(y) m(s)dyds.
\eal
\ee
Notice that
\[
\al
\int_\M G(x, t; y, s) V(y)dy&\leq ||V||_{L^p}\left(\int_\M G^{\frac{p}{p-1}}(x,t;y,s)dy\right)^{(p-1)/p}\\
&= ||V||_{L^p}\left(\int_\M G^{\frac{1}{p-1}}\cdot Gdy\right)^{(p-1)/p}\\
&\leq \sigma\hat C(t)^{1/p} \frac{1}{(t-s)^{\frac{n}{2p}}}.
\eal
\]
Therefore, \eqref{m(t)1} can be further written as
\[
\al
m(t)
&\leq 1+ 2(a-1)\sigma\hat C(T)^{1/p}\int^{t-\e}_0\frac{1}{(t-s)^{\frac{n}{2p}}} m(s)ds+2(a-1)\sigma\hat C(T)^{1/p}m(t)\e^{1-\frac{n}{2p}}.
\eal
\]

Moving the third term on the right hand side to the left hand side, we get
\be
\begin{aligned}\label{m(t)}
\left[1-2(a-1)\sigma\hat C(T)^{1/p}\e^{1-\frac{n}{2p}}\right]m(t)&\leq 1+ 2(a-1)\sigma\hat C(T)^{1/p}\int^{t-\e}_0\frac{1}{(t-s)^{\frac{n}{2p}}} m(s)ds\\
&\leq 1+ 2(a-1)\sigma\hat C(T)^{1/p}\e^{-\frac{n}{2p}}\int^{t-\e}_0m(s)ds.
\end{aligned}
\ee

Setting $$\e=\left(4(a-1)\sigma\hat C(T)^{1/p}\right)^{-\frac{2p}{2p-n}},$$ we have
$$1-2(a-1)\sigma\hat C(T)^{1/p}\e^{1-\frac{n}{2p}}=\frac{1}{2}.$$
Therefore, \eqref{m(t)} becomes
\begin{align*}
m(t)\leq 2+ \left(4(a-1)\sigma\hat C(T)^{1/p}\right)^{\frac{2p}{2p-n}}\int^{t}_0m(s)ds,
\end{align*}
which again is the Gr\"onwall inequality.

Hence, we have
\be
m(t)\leq 2\exp\left\{\left(4(a-1)\sigma\hat C(T)^{1/p}\right)^{\frac{2p}{2p-n}}t\right\},
\ee
for all $t\in[0,T]$. Especially, we have
\be
w(x,t)\leq 2\exp\left\{\left(4(a-1)\sigma\hat C(t)^{1/p}\right)^{\frac{2p}{2p-n}}t\right\},
\ee
for any $t\in[0,\infty)$, i.e.,
$$2^{-1/(a-1)}e^{-(a-1)^{\frac{n}{2p-n}}\left[4\sigma\hat C(t)^{1/p}\right]^{\frac{2p}{2p-n}}t}\leq J(x,t)\leq 1.$$

Therefore, $J$ exists for $t\in [0, \infty)$. This completes the proof of the claim \ref{claim}.\\

Now we continue with the proof of part (1).
In (\ref{ineqtQ}), choosing $J$ as in Claim \ref{claim}, then we deduce
\be
(\Delta  - \p_t) ( t Q )
+ 2 \frac{\d u}{u}\cdot \d (t Q)
\ge  \a t (2 J - \delta J) \frac{1}{n}  \left( | \d  f |^2 - \p_t f \right)^2    - \delta \a  t J|\d f|^4 -Q
\ee For any $T>0$, let $(x_0, t_0)$ be a maximum point of
$t Q = t \left( \a J \frac{ |\d u|^2}{u^2} - \frac{\p_t u}{u} \right)$ in ${\M} \times [0, T]$.
Then at this point, the above inequality induces
\[
0 \ge \a t (2 J - \delta J) \frac{1}{n}  \left( | \d  f |^2 - \p_t f \right)^2
  - \delta \a  t J|\d f|^4 -Q.
\]Clearly we can assume $Q \ge 0$ at $(x_0, t_0)$ since the result is already proven
otherwise.  Then
\[
\left( | \d  f |^2 - \p_t f \right)^2 \ge
 \left( \a J \frac{ |\d u|^2}{u^2} - \frac{\p_t u}{u} \right)^2 + (1-\a J)^2 | \d f|^4.
\]Plugging this into the previous inequality, we find that
\[
0 \ge \a \frac{2 J - \delta J}{n} t Q^2 + \left[  \frac{2-\delta }{n} (1-\a J)^2 - \delta  \right]
\a t J| \d f|^4 - Q.
\] By choosing
\be\lab{alphadelta}
\delta=\frac{2(1-\alpha)^2}{n+(1-\alpha)^2},
\ee one has
\be
\lab{conddelta}
\frac{2-\delta}{n} (1-\a )^2 - \delta  = 0.
\ee
Since $J\leq 1$, we derive from above that
$$\frac{2-\delta}{n} (1-\a J)^2 - \delta  \ge 0\quad \text{on} \quad {\M} \times [0, \infty).$$

Therefore, at $(x_0, t_0)$,
\[
0 \ge \a \frac{2 J - \delta J}{n} t^2 Q^2  - t Q,
\]which infers
\[
t Q \le t Q |_{(x_0, t_0)} \le \frac{n}{(2 - \delta ) \a J}\le\frac{n}{(2 - \delta ) \a \underline{J}(T)},
\]i.e.,
\be
\lab{lyaJ}
\a \underline{J}(t) \frac{ |\d u|^2}{u^2} - \frac{\p_t u}{u} \le \frac{n}{(2  - \delta) \a \underline{J}(t)} \frac{1}{t}.
\ee This proves part (1) of the theorem.\\

For part (2), we first prove an improved short time bound for $J$.

Consider the closed ball in $L^\infty({\M} \times [0, T_0])$
\be
X = \{ w  \in L^\infty({\M} \times [0, T_0])
 \, | \,   1\leq w \le 1+\eta\}.
\ee Here $\eta$ is a positive number in $(0, 1)$, and $T_0$ is a constant to be determined.   Let $w_0=w(\cdot, 0)=1$, and $P$ the map on $X$
\be
\lab{defforP}
P w =  w_0 +2(a-1) \int^t_0 \int_\M G(x, t; y, s) V(y) w(y, s) dyds.
\ee

For any $w \in X$, since $w\geq w_0=1$, we have
$$Pw\geq w_0.$$

Moreover,
\be\label{1}
\al
P w -w_0
&\le 2(a-1)  \int^t_0 \int_\M G(x, t; y, s)  V(y) w(y, s) dyds\\
&\le 2(a-1)(1+\eta) w_0\int^t_0 \int_\M G(x, t; y, s)  V(y)   dyds .
\eal
\ee

Notice that, under the condition (b), we have $\displaystyle \sup_x \int_\M \frac{| Ric^-(y)|^2} {d(x, y)^{n-2}} dy< \infty$. Then by using the Gaussian upper bound of $G$,

\[
\al
&\int^t_0 \int_M G(x, t; y, s)  V(y)   dyds \\
\le&  \left(\int^t_0 \int_\M G(x, t; y, s) dyds \right)^{1/2} \,  \left(\int^t_0 \int_\M G(x, t; y, s) V^2(y) dy \right)^{1/2}\\
\le& C \sqrt{t} \left(\int^t_0 \int_\M \frac{1}{(t-s)^{n/2}} e^{- c d^2(x, y)/(t-s)} V^2(y) dyds \right)^{1/2}\\
\le& C \sqrt{t} \left(\int_\M \frac{1}{d^{n-2}(x, y)}  V^2(y) dy\right)^{1/2} \\
=& C \sqrt{t} \sqrt{\sigma} \equiv C_0 \sqrt{t}.
\eal
\]Hence
\be\label{2}
\int^t_0 \int_\M G(x, t; y, s)  V(y)   dyds \le C_0   \sqrt{t}.
\ee

If $|Ric^-| \in L^p$ with $p>n/2$, the $1/2$ power on $t$ on the right hand side above should be replaced by $1-\frac{n}{2p}$. Here is a quick proof. By Remark 1.2, the heat kernel $G$ also has an
Gaussian upper bound and $|B(x, r)| \le C r^n$. So
\be
\lab{3}
\al
&\int^t_0 \int_\M G(x, t; y, s)  V(y)   dyds \\
\le& C \int^t_0 \int_\M \frac{1}{(t-s)^{n/2}} e^{- c d^2(x, y)/(t-s)} V(y) dyds\\
\le& C  \int^t_0 \left( \int_\M \frac{1}{(t-s)^{np/[2(p-1)]}} e^{- c p/(p-1) \frac{d^2(x, y)}{t-s}}  dy \right)^{(p-1)/p} ds  \, \Vert V \Vert_{L^p}\\
\le& C_0 t^{1-\frac{n}{2p}} .
\eal
\ee

In the following, we prove the theorem under the condition (b), so that \eqref{2} holds. The proof
under the condition (a) works verbatim after replacing \eqref{2} by \eqref{3}.

From \eqref{1} and \eqref{2}, we see that
\[
P w -w_0  \le  C_0(a-1) \sqrt{t} w_0.
\]If we choose
\be
\lab{T0<}
T_0 = \left[ C_0 (a-1) \right]^{-2} \eta^2=\left[ C_0 (5\delta^{-1}-1) \right]^{-2} \eta^2,
\ee then
\be
\lab{Pxx}
P w -w_0  \le \eta  w_0.
\ee Thus $P$ maps $X$ into $X$.

Next we show that $P$ is a contraction mapping on $X$ when $T_0$ is chosen as in \eqref{T0<}.
Let $w_1$ and $w_2$ be two elements in $X$.  Then (\ref{defforP}) implies
\[
\al
&|P w_2 - P w_1|(x, t) =
2(a-1)\left|  \int^t_0 \int_\M G(x, t; y, s) V(y)
\left( w_2 - w_1\right)(y, s) dyds \right|\\
&\le  2(a-1)\int^t_0 \int_\M G(x, t; y, s) V(y) dyds \,
\Vert w_2-w_1 \Vert_\infty\\
& \le C_0 (a-1)\sqrt{t}
\Vert w_2-w_1 \Vert_\infty.
\eal
\] By \eqref{T0<}, we know that under condition (b) of the theorem,
(\ref{Pxx}) holds  and also
\be
\lab{Pcontract}
\Vert P w_2 - P w_1 \Vert_\infty \le \eta   \Vert w_2-w_1 \Vert_\infty.
\ee  Hence $P$ is a contraction map from $X$ to $X$.  The unique fixed point, named
$w$,  is a solution to (\ref{inteqforw}) and (\ref{eqforw}). By the definition of $X$,
we already know that on ${\M} \times [0, T_0]$,
\[
1 \le w \le  1+\eta.
\] From the relations (\ref{defw}), we know that
\[
J = w^{\frac{1}{a-1}}=w^{-\frac{\delta}{5-\delta}}.
\]Hence
\be
\lab{<J<2}
(1+\eta)^{-\frac{\delta}{5-\delta}}\le  J \le    1.
\ee

Let
\be
\lab{beta}
\beta=\frac{n\a}{n+(1-\a)^2} (1+\eta)^{-\frac{\delta}{5-\delta}}.
\ee
Then, \eqref{lyaJ} can be rewritten as
\[\beta\frac{n+(1-\a)^2}{n}\frac{|\nabla u|^2}{u^2}-\frac{\partial_t u}{u}\leq \frac{n}{2\beta}\frac{1}{t},\]
which obviously implies \eqref{eqLY2}.

Moreover, from \eqref{T0<}, \eqref{alphadelta}, and \eqref{beta}, we see that
\be
\lab{T0=3}
T_0 = c (1-\beta)^4
\ee under condition (b) of the theorem.

Similarly, under condition (a), one can get
\be
T_0=[C_0(5\delta-1)]^{-2p/(2p-n)}\eta^{2p/(2p-n)},
\ee
and hence
\be
\lab{T0=4}
T_0 = c (1-\beta)^{4p/(2p-n)}.
\ee

\qed

\medskip

\section{Ricci flow case}

In this section, we consider the Li-Yau bound in the Ricci flow case and prove Theorem \ref{thmLYRF}. The main tool is still the maximum principle applied on a differential inequality involving
Li-Yau type quantity. However, due to the Ricci flow, extra terms involving the Ricci curvature and
Hessian of the solution will come out. In order to proceed we need to create a new term with
the scalar curvature in the denominator.

Before proving the theorem, we carry out some basic computations.

\begin{lemma}\label{F}
Let
\be\label{defF}
F=-\Delta u+\delta\frac{|\nabla u|^2}{u}-\alpha Ru+\frac{\beta u}{R+C},
\ee and operator $\mathcal{L}=\Delta -\frac{\partial }{\partial t}$, where $\delta$, $\alpha$, and $\beta$ are arbitrary constants and $C$ is a constant so that $R+C>0$. Then
\be\lab{eqF}
\begin{aligned}
\mathcal{L}F&= \frac{1}{u}\left|u_{ij}-\frac{u_iu_j}{u}+uR_{ij}\right|^2+\frac{2\delta-1}{u}\left|u_{ij}-\frac{u_iu_j}{u}\right|^2+\frac{1}{(2\alpha-1)u}\left|(2\alpha-1)uR_{ij}+\frac{u_iu_j}{u}\right|^2\\
&\quad -\frac{1}{2\alpha-1}\frac{|\nabla u|^4}{u^3}+\frac{\alpha u}{R+C}\left|\nabla R-\frac{R+C}{u}\nabla u\right|^2+(\frac{\beta}{(R+C)^3}-\frac{\alpha}{R+C})u|\nabla R|^2\\
&\quad -\frac{\alpha (R+C)|\nabla u|^2}{u}+\frac{2\beta u|R_{ij}|^2}{(R+C)^2}+\frac{\beta u}{R+C}\left| \frac{\nabla R}{R+C}-\frac{\nabla u}{u}\right|^2-\frac{\beta|\nabla u|^2}{u(R+C)}.
\end{aligned}
\ee
\end{lemma}

\proof

It follows from \eqref{ricciheateq}  that
\begin{align}\label{delta u}
\mathcal{L}(\Delta u)=-2R_{ij}u_{ij},
\end{align}
and
\begin{align}\label{nabla u}
\mathcal{L}(|\nabla u|^2)=2|u_{ij}|^2.
\end{align}
Also, it is well known that under the Ricci flow we have
\begin{equation}\label{R}
\mathcal{L}R=-2|R_{ij}|^2.
\end{equation}
On the other hand, it is straightforward to check that for any smooth functions $f$ and $g$, one has
\begin{align}\label{f/g}
\mathcal{L}(\frac{f}{g})=\frac{1}{g}\mathcal{L}f-\frac{f}{g^2}\mathcal{L}g-\frac{2}{g}\nabla_i\frac{f}{g}\nabla_i g,
\end{align}
and
\begin{align}\label{fg}
\mathcal{L}(fg)=f\mathcal{L}g+g\mathcal{L}f+2\nabla_i f\nabla_i g.
\end{align}
It then follows from \eqref{ricciheateq}, \eqref{nabla u}, \eqref{R}, \eqref{f/g} and \eqref{fg} that
\begin{equation}\label{nabla u/u}
\begin{aligned}
\mathcal{L}\left(\frac{|\nabla u|^2}{u}\right)&=\frac{1}{u}\mathcal{L}|\nabla u|^2-\frac{2}{u}\nabla_i\frac{|\nabla u|^2}{u}\nabla_i u\\
&=\frac{2}{u}|u_{ij}|^2-\frac{4}{u^2}u_{ij}u_iu_j+\frac{2|\nabla u|^4}{u^3}\\
&=\frac{2}{u}|u_{ij}-\frac{u_iu_j}{u}|^2,
\end{aligned}
\end{equation}

\begin{equation}\label{Ru}
\mathcal{L}(Ru)=u\mathcal{L}R+2\nabla_iR\nabla_iu=-2u|R_{ij}|^2+2\nabla_iR\nabla_i u,
\end{equation}
and
\begin{equation}\label{u/R+C}
\begin{aligned}
\mathcal{L}\left(\frac{u}{R+C}\right)&=-\frac{u}{(R+C)^2}\mathcal{L}R-\frac{2}{R+C}\nabla_i\frac{u}{R+C}\nabla_iR\\
&=\frac{2u|R_{ij}|^2}{(R+C)^2}+\frac{2u|\nabla R|^2}{(R+C)^3}-\frac{2}{(R+C)^2}\nabla_i R\nabla_i u.
\end{aligned}
\end{equation}

Thus, by \eqref{defF}, \eqref{delta u}, \eqref{nabla u/u}, \eqref{Ru} and \eqref{u/R+C}, we have, after splitting zeros in four occasions, that
\begin{align*}
\mathcal{L}F&=2R_{ij}u_{ij}+\frac{2\delta}{u}\left|u_{ij}-\frac{u_iu_j}{u}\right|^2+2\alpha u|R_{ij}|^2-2\alpha\nabla_iR\nabla_iu\\
& \quad +\frac{2\beta u|R_{ij}|^2}{(R+C)^2}+\frac{2\beta u|\nabla R|^2}{(R+C)^3}-\frac{2\beta}{(R+C)^2}\nabla_i R\nabla_i u\\
&=\frac{1}{u}\left|u_{ij}-\frac{u_iu_j}{u}+uR_{ij}\right|^2+\frac{2\delta-1}{u}\left|u_{ij}-\frac{u_iu_j}{u}\right|^2+(2\alpha-1)u|R_{ij}|^2+2R_{ij}\frac{u_iu_j}{u}\\
&\quad +\frac{1}{2\alpha-1}\frac{|\nabla u|^4}{u^3}-\frac{1}{2\alpha-1}\frac{|\nabla u|^4}{u^3}-2\alpha\nabla_iR\nabla_iu+\frac{\alpha u|\nabla R|^2}{R+C}+\frac{\alpha (R+C)|\nabla u|^2}{u}\\
&\quad +(\frac{\beta}{(R+C)^3}-\frac{\alpha}{R+C})u|\nabla R|^2-\frac{\alpha (R+C)|\nabla u|^2}{u}+\frac{2\beta u|R_{ij}|^2}{(R+C)^2}-\frac{2\beta}{(R+C)^2}\nabla_i R\nabla_i u\\
&\quad +\frac{\beta u|\nabla R|^2}{(R+C)^3}+\frac{\beta|\nabla u|^2}{u(R+C)}-\frac{\beta|\nabla u|^2}{u(R+C)}.
\end{align*}
Observe that the 3rd, 4th and 5th terms, 7th, 8th and 9th terms and 13th, 14th and 15th
terms form complete squares, respectively.  Hence we get \eqref{eqF}. \qed\\

Now we are ready to prove Theorem \ref{thmLYRF}.\\

\hspace{-.5cm}{\bf Proof of Theorem \ref{thmLYRF}:}
Assume that $1>\delta \geq \frac{1}{2}$ and $\alpha>1$. By choosing $C=2$ and $\beta= \alpha (R_1+2)^2$ in the above lemma, we have
\begin{equation}\label{LF1}
\begin{aligned}
\mathcal{L}F&\geq \frac{1}{nu}\left|\Delta u-\frac{|\nabla u|^2}{u}+uR\right|^2+\frac{2\delta-1}{nu}\left|\Delta u-\frac{|\nabla u|^2}{u}\right|^2 -\frac{1}{2\alpha-1}\frac{|\nabla u|^4}{u^3}\\
&\quad -\frac{\alpha (R+2)|\nabla u|^2}{u}-\frac{\beta|\nabla u|^2}{u(R+2)}.
\end{aligned}
\end{equation}
Notice that
$$\Delta u-\frac{|\nabla u|^2}{u}=\left(\Delta u-\frac{|\nabla u|^2}{u}+uR\right)-uR.$$
We rewrite \eqref{LF1} as
\begin{align*}
\mathcal{L}F&\geq\frac{2\delta}{nu}\left|\Delta u-\frac{|\nabla u|^2}{u}+uR\right|^2-\frac{2(2\delta-1)R}{n}\left(\Delta u-\frac{|\nabla u|^2}{u}+uR\right) +\frac{(2\delta-1)R^2u}{n}\\
&\quad -\frac{1}{2\alpha-1}\frac{|\nabla u|^4}{u^3}-\frac{\alpha (R+2)|\nabla u|^2}{u}-\frac{\beta|\nabla u|^2}{u(R+2)}.
\end{align*}
According to the definition of $F$ in \eqref{defF}, the above inequality becomes
\begin{align*}
\mathcal{L} F &\geq\frac{2\delta}{nu}\left|F+(1-\delta)\frac{|\nabla u|^2}{u}+(\alpha-1)Ru-\frac{\beta u}{R+2}\right|^2\\
&\quad +\frac{2(2\delta-1)R}{n}\left(F+(1-\delta)\frac{|\nabla u|^2}{u}+(\alpha-1)Ru-\frac{\beta u}{R+2}\right)\\
&\quad +\frac{(2\delta-1)R^2u}{n}-\frac{1}{2\alpha-1}\frac{|\nabla u|^4}{u^3}
-\left[\alpha(R+2)+\frac{\beta}{R+2}\right]\frac{|\nabla u|^2}{u}.
\end{align*}
Let $Q=tF-\theta u$. Then at $t=0$, we have $Q<0$. Suppose that at time $t_0>0$ and point $x_0\in \M$, $Q$ reaches $0$ for the first time. Then at $(x_0, t_0)$, we have $t_0 F=\theta u$ and

\begin{align*}
0&\geq t_0\mathcal{L}Q(x_0,t_0)\\
&\geq-\theta u +\frac{2\delta}{nu}\left|\theta u+(1-\delta)\frac{|\nabla u|^2}{u}t_0+(\alpha-1)Rut_0-\frac{\beta ut_0}{R+2}\right|^2\\
&\quad +\frac{2(2\delta-1)Rt_0}{n}\left(\theta u+(1-\delta)\frac{|\nabla u|^2}{u}t_0+(\alpha-1)Rut_0-\frac{\beta u t_0}{R+2}\right)\\
&\quad +\frac{(2\delta-1)R^2u}{n}t_0^2-\frac{1}{2\alpha-1}\frac{|\nabla u|^4}{u^3}t_0^2 -\left[\alpha(R+2)+\frac{\beta}{R+2}\right]\frac{|\nabla u|^2}{u}t_0^2
\end{align*}  After expanding the first square, we deduce
\begin{align*}
&0 \ge -\theta u+\frac{2\delta}{n}\theta^2u+\frac{2\delta(1-\delta)^2}{n}\frac{|\nabla u|^4}{u^3}t_0^2+\frac{2\delta(\alpha-1)^2R^2u}{n}t_0^2\\
&\quad +\frac{2\delta\beta^2u}{n(R+2)^2}t_0^2+\frac{4\delta(1-\delta)\theta}{n}\frac{|\nabla u|^2}{u}t_0+\frac{4\delta(\alpha-1)\theta Ru}{n}t_0-\frac{4\delta\beta \theta u}{n(R+2)}t_0\\
&\quad+\frac{4\delta(1-\delta)(\alpha-1)R}{n}\frac{|\nabla u|^2}{u}t^2_0-\frac{4\delta(1-\delta)\beta}{n(R+2)}\frac{|\nabla u|^2}{u}t_0^2-\frac{4\delta(\alpha-1)\beta R u}{n(R+2)}t_0^2\\
&\quad +\frac{2(2\delta-1)\theta Ru}{n}t_0+\frac{2(2\delta-1)(1-\delta)R}{n}\frac{|\nabla u|^2}{u}t_0^2+\frac{2(2\delta-1)(\alpha-1)R^2u}{n}t_0^2\\
&\quad -\frac{2(2\delta-1)\beta R u}{n(R+2)}t_0^2+\frac{(2\delta-1)R^2u}{n}t_0^2-\frac{1}{2\alpha-1}\frac{|\nabla u|^4}{u^3}t_0^2 -\left[\alpha(R+2)+\frac{\beta}{R+2}\right]\frac{|\nabla u|^2}{u}t_0^2.
\end{align*} This becomes, after combining similar terms,
\begin{align*}
&0 \geq -\theta u+\frac{2\delta}{n}\theta^2u-\frac{4\delta\beta \theta u}{n(R+2)}t_0-\frac{(4\delta\alpha-2)\theta u}{n}t_0-\frac{(4\delta\alpha-2)\beta R_1u}{n(R+2)}t_0^2 \\
&\quad +\left(\frac{2\delta(1-\delta)^2}{n}-\frac{1}{2\alpha-1}\right)\frac{|\nabla u|^4}{u^3}t_0^2\\
&\quad +\left(\frac{4\delta(1-\delta)\theta}{n}t_0-\frac{4\delta(1-\delta)\beta}{n(R+2)}t_0^2-\frac{2\beta}{(R+2)} t_0^2-\frac{(4\delta\alpha-2)(1-\delta)}{n}t_0^2\right)\frac{|\nabla u|^2}{u}.
\end{align*}

It is straightforward to check that by choosing
$$\alpha=\frac{n}{2\delta(1-\delta)^2}, \qquad \textrm{and}\qquad \theta=\frac{n}{2\delta}+\frac{4n\beta T}{\delta(1-\delta)}$$
one has
\begin{equation}
\frac{2\delta(1-\delta)^2}{n}-\frac{1}{2\alpha-1}>0,
\end{equation}
\begin{equation}
\frac{4\delta(1-\delta)\theta}{nT}\geq \left[\frac{4\delta(1-\delta)}{n}+2\right] \beta+\frac{(4\delta\alpha-2)(1-\delta)}{n},
\end{equation}
and
\begin{equation}
\frac{2\delta}{nT^2}\theta^2-(\frac{1}{T^2}+\frac{4\delta\beta }{nT}+\frac{4\delta\alpha }{nT})\theta-\frac{4\delta\alpha\beta R_1 }{n}>0.
\end{equation}

Therefore, we have a contradiction. It follows that
$$-\Delta u+\delta\frac{|\nabla u|^2}{u}-\alpha Ru+\frac{\beta u}{R+2}\leq\frac{\theta u}{t}$$
for any $t\in(0, T)$, which is \eqref{LYRF}.\qed\\

In general, along the Ricci flow we have $$-\sup_\M R(x,0)\leq R(x,t)\leq \sup_{\M \times[0,T)} R(x,t).$$ Denote by $R_1=sup_{\M \times[0,T)}R(x,t)$ and
\begin{equation*}
R_0=\left\{\begin{aligned}&\sup_\M R^{-}(x,0),\ if\  \sup_\M R^{-}(x,0)>0\\
&\inf_\M R(x,0), \ if\  \sup_\M R^{-}(x,0)=0.
\end{aligned}
\right.
\end{equation*}

It is not hard to check that by choosing $C=2R_0$ and $\beta=\alpha(R_1+2R_0)^2$ in Lemma \ref{F} and repeating the proof of Theorem \ref{thmLYRF}, we can get the following scaling invariant Li-Yau bounds.

\begin{theorem}\label{thmLYRF1}
Let $\M$ be a compact $n$-dimensional Riemmanian manifold, and $g_{ij}(t)$, $t\in [0,T)$, a solution of the Ricci flow \eqref{RF1} on $\M$. Suppose that $u$ is a positive solution of the heat equation \eqref{ricciheateq}. Then, for any $\delta\in[\frac{1}{2},1)$, when $\sup_\M R^-(x,0)>0$, we have
\begin{equation}
\delta\frac{|\nabla u|^2}{u^2}-\frac{\partial_t u}{u}-\alpha R+\frac{\beta}{R+2R_0}\leq\frac{\theta}{t},
\end{equation}
and when $\sup_\M R^-(x,0)=0$, we have
\begin{equation}
\delta\frac{|\nabla u|^2}{u^2}-\frac{\partial_t u}{u}-\alpha R+\frac{\beta}{R}\leq\frac{\theta}{t},
\end{equation}
for any $t\in(0,T)$, where $\alpha=\frac{n}{2\delta(1-\delta)^2}$, $\beta=\alpha (R_1+2R_0)^2$, and $\theta=\frac{n}{2\delta}+\frac{4n\beta T}{\delta(1-\delta)R_0}$.
\end{theorem}

\begin{remark}
From \eqref{BZ}, one can see that if $R>0$, then there are both a forward inequality $u_t\geq -\frac{Ba}{t}$, and a backward inequality $u_t\leq\frac{Ba}{t}$.

The Li-Yau bound \eqref{LYRF} obtained above gives us a stronger forward Harnack inequality $\frac{u_t}{u}\geq -\frac{k}{t}$ when $R>0$. However, it seems that a backward Harnack inequality of the form
$\frac{u_t}{u}\leq \frac{k}{t}$ cannot be expected. Because if this were the case, then one would have
$u(x,t_2)\leq u(x,t_1)(\frac{t_2}{t_1})^k$. Now suppose that $M$ is an Einstein manifold $R_{ij}=\rho g_{ij}$ with $\rho>0$ and $u(x,t)=G(x,t;x_0,0)$ the heat kernel under the Ricci flow. According to a result in \cite{CZ}, we have the Gaussian lower and upper bounds of $G$, i.e.,
$$Ct^{-\frac{n}{2}}e^{-\frac{cd^2_{t}(x,x_0)}{t}}\leq G(x,t;x_0,0)\leq Ct^{-\frac{n}{2}}e^{-\frac{d_t^2(x,y)}{ct}}.$$
It then follows that
$$Ct_2^{-\frac{n}{2}}e^{-\frac{cd^2_{t_2}(x,x_0)}{t_2}}\leq G(x,t_2;x_0,0)\leq G(x,t_1;x_0,0)(\frac{t_2}{t_1})^k\leq C(\frac{t_2}{t_1})^kt_1^{-\frac{n}{2}}e^{-\frac{d_{t_1}^2(x,y)}{ct_1}},$$
i.e.,
$$e^{-\frac{(1-2\rho t_2)d^2_0(x,x_0)}{t_2}+\frac{(1-2\rho t_1)d^2_0(x,x_0)}{t_1}}\leq C(\frac{t_2}{t_1})^{\frac{n}{2}+k}.$$
Obviously, when $t_2=2t_1$ and $x\neq x_0$, we get a contradiction for $t_1$ small enough.
\end{remark}
\medskip

\newcommand\redsout{\bgroup\markoverwith{\textcolor{red}{\rule[0.5ex]{3pt}{0.5pt}}}\ULon}
\section{Applications on extending Colding-Naber result}

In this section, we mainly apply the Li-Yau bound \eqref{eqLY2} to extend parabolic approximations of distance functions of Colding-Naber \cite{CoNa} to the case where $|Ric^-|\in L^{p}$ for some $p>\frac{n}{2}$. This result is a moderate variation of the one in Tian-Z.L. Zhang \cite{TZz:1}. In addition, some of the intermediate results are also proved by replacing the condition that $|Ric^-|\in L^{p}$ by $|Ric|\in K^{2,n-2}$, where $K^{p,\lambda}$ denotes the Kato type space with the norm
$$\|w\|_{K^{p,\lambda}}=\left(\sup_\M\int_\M \frac{|w|^p}{d^{\lambda}(x,y)}dy\right)^{\frac{1}{p}}.$$ We hope that these results will find applications in the study of K\"ahler Ricci flows  which enjoy the property that the condition $|Ric|\in K^{2,n-2}$ is preserved (\cite{TZq2}).

Let $(\M, g_{ij})$ be a compact $n$-dimensional Riemannian manifold.
We list below the following four assumptions, part of which will be used in various results in the section.  In particular, we will present two theorems under the assumptions
\textit{A1} and \textit{A2};  and under the assumptions \textit{A1}, \textit{A3} and \textit{A4} respectively. As explained below, these conditions correspond to the normalized K\"ahler-Ricci flow of dimensions 3 or less and to all dimensions respectively.\\

\textit{A1}: $M$ is $\kappa$-noncollapsed for some constant $\k$, i.e.,
\begin{equation}\label{noncollapsing}
\v(B_r(x))\geq \k r^n,\quad \forall x\in \M,\ \textrm{and}\ r\leq 1.
\end{equation}

\textit{A2}: $||Ric^-||_{L^p}\leq \Lambda$ for some $p>\frac{n}{2}$.

\textit{A3}: $||Ric||_{K^{2,n-2}}\leq \Gamma$, the heat kernel of \eqref{heatequation}
has a Gaussian upper bound for $0<t \le 1$ as in \eqref{GUP}, and  $|B(x, r)| \le C r^n,
\quad \forall x\in \M,\ \textrm{and}\ r\leq 1$.

\textit{A4}: on any geodesic ball $B_r(x)$, there exists a function $\phi\in C_0^{\infty}(B_r(x))$ such that
$\phi\geq0,\ \phi=1\ \textrm{in}\ B_{r/2}(x),$
and
$|\nabla \phi|^2+|\Delta \phi|\leq Cr^{-2}.$\\

\begin{remark}
\lab{rkgoodcut}
Under assumptions \textit{A1} and \textit{A2}, the existence of cut-off functions as in $\textit{A4}$ was first proved by Petersen-Wei \cite{PW2}, where the volume doubling property and laplacian comparison theorem were used. It is known that in  that \textit{A1} (\cite{P:1}) and \textit{A2} (\cite{TZz:1}) hold uniformly on each time slice of the normalized K\"ahler-Ricci flow of complex dimension $3$ and less.

On the other hand, as shown in \cite{BZ} Theorem 1.3 for the  Ricci flow, given the Gaussian upper and lower bound of the heat kernel $G(x,t;y,0)$ and its time derivative, one can also construct a cut-off function $\phi$ such that $0<\phi\leq 1$ in $B_r(x)$, $\phi\geq c>0$ in $B_{r/2}(x)$, and $|\Delta\phi|+|\d\phi|^2\leq Cr^{-2}$ for any $x\in\M$ and $r\leq r_0$. And hence a cut-off function in an annulus as in Lemma \ref{cutoffannulus} can be obtained, and used instead. Although that paper dealt with Ricci flow case, the same method works for fixed manifolds since the cut-off function is constructed from the heat kernel. Moreover, by a simple covering argument, after composition with a one variable function, one can further refine the
cut-off function so that $\phi=1$ in a smaller ball with, say, half of the radius. i.e. It becomes a "good" cut-off function.
\end{remark}

\begin{remark} We mention that conditions \it{A1}, \it{A3}, \it{A4} hold uniformly on each time slice of the normalized K\"ahler-Ricci flow in all dimensions. Indeed, \it{A1} is Perelman's $\kappa$ noncollapsing (\cite{P:1}).
From the papers \cite{TZq1} and \cite{TZq2}, we know \it{A3} holds on each time slice of the normalized K\"ahler-Ricci flow.
In \cite{TZq1}, during the proof of Lemma 2.3, a Gaussian upper and lower bound for the stationary heat kernel of each time slice is proven. Hence by Remark \ref{rkgoodcut}, condition \it{A4} holds, namely a good cut-off function exists.

If one works a little harder, by assuming just the Gaussian upper bound, one can prove the Gaussian lower bound holds under the condition $||Ric||_{K^{2,n-2}}\leq \Gamma$. So {\it A1} and
{\it A3} together actually imply {\it A4}. But we will not seek this generalization here.

\end{remark}

Let $h^{\pm}_t(x)$ be the parabolic approximations of the local distance functions as defined in \eqref{defh}.

The main results of this section are Theorem \ref{distance approximation} and
Theorem \ref{distance approximation 2} below.
As remarked above, the first one works for  the normalized K\"ahler-Ricci flow of dimension 3 and less. The second works for  the normalized K\"ahler-Ricci flow of all dimensions, but the result is weaker in that the bounds are less concrete.

\begin{theorem}\label{distance approximation}
Assume that A1 and A2 are satisfied. Let $O^+$ and $O^-$ be two fixed points in $\M$. Denote by $d_0=d(O^+, O^-)$. Then for some fixed $\delta>0$, there exist constants $C=C(n,p,\k,\Lambda,\delta)$ and $\overline{\e}=\overline{\e}(n,p,\delta)$, such that for any $0<\e\leq\overline{\e}$,
\[
x\in M_{\delta, 2} \equiv \{x \in \M \, | \,
\delta d_0 < d(x, \{O^+, O^-\}) \le 2 d_0 \}
\] with
\[
e(x) \equiv d(O^-,x)+d(O^+,x)-d(O^+,O^-) \leq \e^2d_0,
\] and any $\e$-geodesic $\sigma: [0, d_0]\rightarrow \M$ connecting $O^+$ and $O^-$, there exists $r\in[\frac{1}{2}, 2]$ satisfying\\
(1) $\displaystyle \left|h^{\pm}_{r{\e}^2d_0^2}-d^{\pm}\right|\leq C d_0(\e^2+\e^{2-\frac{n}{2p}})$.\\
(2) $\displaystyle \oint_{B_{{\e}d_0}(x)}\left||\nabla h^{\pm}_{r{\e}^2d_0^2}|^2(y)-1\right|dy\leq C(\e+\e^{1-\frac{n}{2p}})$.\\
(3) $\displaystyle \oint_{\delta d_0}^{(1-\delta)d_0}\oint_{B_{{\e}d_0}(\sigma(s))}\left||\nabla h^{\pm}_{r{\e}^2d_0^2}|^2(y)-1\right|dyds\leq C(\e^2+\e^{2-\frac{n}{p}})$.\\
(4) $\displaystyle \oint_{\delta d_0}^{(1-\delta)d_0}\oint_{B_{{\e}d_0}(\sigma(s))}\left|\nabla^2 h^{\pm}_{r{\e}^2d_0^2}\right|^2(y)dyds\leq \frac{C(1+\e^{-\frac{n}{p}})}{d_0^2}$.

\end{theorem}

More explanations of the notations in the above theorem can be found in the following context.

\begin{remark} Theorem \ref{distance approximation} was first obtained by Tian-Z. Zhang in \cite{TZz:1}. But the exponents on the right hand sides of (3) and (4) are slightly different (comparing to Theorem 2.25 in \cite{TZz:1}).

\end{remark}

If $|Ric^-|\in K^{2,n-2}$, due to the absence of volume doubling property and heat kernel Gaussian bounds, to prove a similar result as in Theorem \ref{distance approximation}, additional assumptions need to be imposed. More precisely, we can prove

\begin{theorem}\label{distance approximation 2}
Assume that A1, A3, and A4 are satisfied. Then for some fixed $\delta>0$, any $q>0$ and $\lambda>n-2q$, there exist constants $C=C(n,q,\lambda,\k,\Gamma,\delta)$ and $\overline{\e}=\overline{\e}(n,\delta)$, such that for any $0<\e\leq\overline{\e}$,
$x\in M_{\delta, 2}$ with
$e(x)\leq \e^2d_0$, $d_0=d(O^+,O^-)$ and any $\e$-geodesic $\sigma: [0, d_0]\rightarrow \M$ connecting $O^+$ and $O^-$, there exists $r\in[\frac{1}{2}, 2]$ satisfying\\
(1) $\displaystyle \left|h^{\pm}_{r{\e}^2d_0^2}-d^{\pm}\right|\leq C d_0(\e^2+||\psi^{\pm}||_{K^{q,\lambda}}\e^{2-\frac{n-\lambda}{q}}d_0^{2-\frac{n-\lambda}{q}})$.\\
(2) $\displaystyle \oint_{B_{{\e}d_0}(x)}\left||\nabla h^{\pm}_{r{\e}^2d_0^2}|^2(y)-1\right|dy\leq C(\e+||\psi^{\pm}||_{K^{q,\lambda}}\e^{1-\frac{n-\lambda}{q}}d_0^{1-\frac{n-\lambda}{q}})$.\\
(3) $\displaystyle \oint_{\delta d_0}^{(1-\delta)d_0}\oint_{B_{{\e}d_0}(\sigma(s))}\left||\nabla h^{\pm}_{r{\e}^2d_0^2}|^2(y)-1\right|dyds\leq C(\e^2d_0+\e d_0^2+||\psi^{\pm}||_{K^{q,\lambda}}\e^{2-\frac{n-\lambda}{q}}d_0^{2-\frac{n-\lambda}{q}})$.\\
(4) $\displaystyle \oint_{\delta d_0}^{(1-\delta)d_0}\oint_{B_{{\e}d_0}(\sigma(s))}\left|\nabla^2 h^{\pm}_{r{\e}^2d_0^2}\right|^2(y)dyds\leq \frac{C(1+\e^{-1}+||\psi^{\pm}||_{K^{q,\lambda}}\e^{-\frac{n-\lambda}{q}})}{d_0^2}$.

\end{theorem}

In the following, we mainly present the proof of Theorem \ref{distance approximation}. The proof of Theorem \ref{distance approximation 2} is similar. The proof essentially follows the arguments in section 2 of \cite{CoNa}. Thus, in the context below, we first just present most of the corresponding intermediate steps without proofs. For the $|Ric^-|\in L^{p}$ case, we start from the volume comparison theorem proved by Petersen-Wei.

\begin{theorem}{(Petersen-Wei \cite{PW1})}\label{volumecomparison}
If A2 is satisfied, then there exists a constant $C=C(n,p)$ which is nondecreasing in $R$ such that for all $r\leq R$ and $x\in \M$, we have
\begin{equation}
\left(\frac{\v(B_R(x))}{R^n}\right)^{1/2p}-\left(\frac{\v(B_r(x))}{r^n}\right)^{1/2p}\leq C\Lambda^{1/2p}R^{1-\frac{n}{2p}},
\end{equation}
where $B_r(x)$ denotes the geodesic ball centered at $x$ with radius $r$.
\end{theorem}

A very important corollary of the above theorem is the following volume doubling property (see \cite{PW2} Theorem 2.1).

\begin{theorem}{(Petersen-Wei \cite{PW2})}
Given $\alpha<1$ and $p>n/2$. Assume that A1 and A2 are satisfied. Then there exists an $R=R(\alpha, p, n, \Lambda)>0$ such that for any $0<r_1\leq r_2\leq R$, we have
\begin{equation}\label{Lpvolumedoubling}
\alpha\frac{r_1^n}{r_2^n}\leq \frac{\textrm{vol}\, B_{r_1}(x)}{\textrm{vol}\, B_{r_2}(x)}.
\end{equation}

\end{theorem}

\begin{remark}
For $|Ric^-|\in K^{2,n-2}$, it is not known whether the volume comparison theorem holds. Thus, we need to assume $|B(x, r)| \le C r^n$ in {\it A3}, which combining with {\it A1} provides the volume doubling property.
\end{remark}

By using the above theorem, Petersen-Wei also obtained the following cut-off function, which was first observed by Cheeger-Colding in \cite{ChCo} for manifolds with Ricci curvature bounded from below.
\begin{lemma}{(Petersen-Wei \cite{PW2})}\label{cutoffball}
Suppose that A1 and A2 are satisfied. There exist $r_0=r_0(n,p,\k,\Lambda)$ and $C=C(n,p,\k,\Lambda)$ such that on any geodesic ball $B_r(x)$, $r\leq r_0$, there exists a function $\phi\in C_0^{\infty}(B_r(x))$ such that
$$\phi\geq0,\quad\ \phi=1\ \textrm{in}\ B_{r/2}(x),$$
and
$$|\nabla \phi|^2+|\Delta \phi|\leq Cr^{-2}.$$
\end{lemma}

Let $E$ be a closed subset of $M$. Denote the $r$-tubular neighborhood of $E$ by
$$T_r(E)=\{x\in \M|\ d(x,E)\leq r\}.$$
For $0<r_1<r_2$, define the annulus $A_{r_1, r_2}(E)=T_{r_2}(E)\setminus \overline{T_{r_1}(E)}$.
Using the lemma above and a similar argument as in the proof of Lemma 2.6 in \cite{CoNa}, one has

\begin{lemma}{(Tian-Z. Zhang \cite{TZz:1})}\label{cutoffannulus}
Suppose that A1 and A2 are satisfied. For any $R>0$, there exists $C=C(n,p,\k,\Lambda, R)$ such that the following holds.
Let $E$ be any closed subset and $0<r_1<10r_2<R$. There exists a function $\phi\in C^{\infty}(B_R(E))$ satisfying
$$\phi\geq0,\quad\ \phi=1\ \textrm{in}\ A_{3r_1,r_2/3}(E),\quad \phi=0\  \textrm{outside}\  A_{2r_1, r_2/2}(E),$$
$$|\nabla \phi|^2+|\Delta \phi|\leq Cr_1^{-2}\  \textrm{in}\ A_{2r_1, 3r_1}(E),$$
and
$$|\nabla\phi|^2+|\Delta\phi|\leq Cr_2^{-2} \ \textrm{in}\ A_{r_2/3, r_2/2}(E).$$
\end{lemma}

Let $G(y,t;x,0)=G(x,t;y,0)$ be the heat kernel on $M$. It can be showed that $G(y,t;x,0)$ has both Gaussian upper and lower bounds as follows
\begin{lemma}{(Tian-Z. Zhang \cite{TZz:1} ) }
\label{HK}
Suppose that A1 and A2  are satisfied. There exist positive constants $C_i=C_i(n,p,\k,\Lambda)$, $i=1,2,3,4$, such that
\begin{equation}\label{HKbound}
C_1t^{-\frac{n}{2}}e^{\frac{-C_2d^2(x,y)}{t}}\leq G(y,t;x,0)\leq C_3t^{-\frac{n}{2}}e^{-\frac{d^2(x,y)}{C_4t}},\ \forall x,y\in \M,\ \textrm{and}\  0<t\leq 1.
\end{equation}
\end{lemma}
Actually, the Gaussian upper bound can be obtained by an $L^1$ mean value inequality for $G(y,t;x,0)$ and Grigor'yan's method in \cite{Gr1997}. Then the lower bound follows from the upper bound and an on-diagonal gradient bound for $G(y,t;x,0)$.

By using Duhamel's principle, it is not hard to prove the following $L^1$ Harnack inequalities.

\begin{lemma}\label{parabolicL1harnack}
Let $u(x,t)$ be a nonnegative function satisfying
\begin{equation}
\frac{\partial u}{\partial t}\geq\Delta u -\xi,
\end{equation}
where $\xi=\xi(x)\geq0$ is a smooth function.

(i) If A1 and A2 are satisfied, then for any $q>\frac{n}{2}$, there exists a constant $C=C(n,p,q,\kappa,\Lambda)$ such that
\begin{equation}\label{L1harnack1}
\oint_{B_r(x)}u(y,0)dy\leq C\left(u(x,r^2)+r^{2-\frac{n}{q}}\|\xi\|_{L^q}\right)
\end{equation}
holds for any $x\in \M$ and $0<r\leq 1$.

More generally, we have
\begin{equation}
\oint_{B_r(x)}u(y,0)dy\leq C\left(\inf_{B_r(x)}u(\cdot,r^2)+r^{2-\frac{n}{q}}\|\xi\|_{L^q}\right).
\end{equation}

(ii) If A1 and A3 are satisfied, then for any $q>0$ and $\lambda>n-2q$, there exists a constant $C=C(n,q,\lambda,\kappa, \Gamma)$ such that
\begin{equation}\label{L1harnack2}
\oint_{B_r(x)}u(y,0)dy\leq C\left(u(x,r^2)+r^{2-\frac{n-\lambda}{q}}\|\xi\|_{K^{q,\lambda}}\right)
\end{equation}
holds for any $x\in \M$ and $0<r\leq 1$.

More generally, we have
\begin{equation}
\oint_{B_r(x)}u(y,0)dy\leq C\left(\inf_{B_r(x)}u(\cdot,r^2)+r^{2-\frac{n-\lambda}{q}}\|\xi\|_{K^{q,\lambda}}\right).
\end{equation}
\end{lemma}

\begin{remark}
Part i) above is in Corollary 2.15 in \cite{TZz:1}. The proof of part ii) is analogous to part i). The heat kernel bounds follow from the Li-Yau gradient bound in Theorem \ref{thmLY} (b).
\end{remark}

Let $O^+$ and $O^-$ be two fixed points in $\M$. Following \cite{CoNa}, define
\begin{equation}
d^-(x)=d(O^-,x),\  d^+(x)=d(O^+,O^-)-d(O^+,x),
\end{equation}
and
\begin{equation}
e(x)=d^-(x)-d^+(x)=d(O^-,x)+d(O^+,x)-d(O^+,O^-).
\end{equation}

First of all, we have in barrier sense that
\begin{align}
\Delta d^-(x)&\leq \frac{n-1}{d^-}+\psi^-\\
-\Delta d^+(x)&\leq  \frac{n-1}{d^+}+\psi^+,
\end{align}
where $\psi^-=\max\{\Delta d^-(x)-\frac{n-1}{d^-},0\}$, and $\psi^+=\max\{-\Delta d^+(x)-\frac{n-1}{d^+},0\}$.
Moreover, it follows from Lemma 2.2 in \cite{PW1} that
\begin{equation}
\int_{B_r(x)}|\psi^{\pm}|^{2p}(y)dy\leq C(n,p)\int_{B_r(x)}|Ric^-|^p(y)dy.
\end{equation}

Denote by
$$d_0=d(O^+,O^-)\ \textrm{and}\ M_{r_1,r_2}=A_{r_1d_0,r_2d_0}(\{O^+, O^-\}).$$
With out loss of generality, we may assume that $d_0\leq1$.

With the preparation above, by applying the method in section 2.1 in \cite{CoNa}, we can now prove

\begin{lemma}\label{excess}
For some fixed $\delta>0$,

i) if A1 and A2 are satisfied, then there exist a small constant $\overline{\e}=\overline{\e}(n,p,\delta)$, and a constant $C=C(n,p,\k,\Lambda,\delta)$ such that for any $0<\e\leq\overline{\e}$, we have
\begin{equation}\label{general excess estimate}
\oint_{B_{\e d_0}(x)}e(y)dy\leq C\left[e(x)+\e^2d_0+(||\psi^+||_{L^{2p}}+||\psi^-||_{L^{2p}})\e^{2-\frac{n}{2p}}d_0\right]\leq C(e(x)+\e^{2-\frac{n}{2p}}d_0),
\end{equation}
for all $x\in M_{\frac{\delta}{4},16}$.

In particular, this implies the excess estimate of Abresch-Gromoll \cite{AbGr}, i.e.,
\begin{equation}\label{excess estimate}
e(y)\leq C\e^{1+\alpha(n,p)}d_0,\ \forall y\in B_{\frac{1}{2}\e d_0}(x)
\end{equation}
whenever $e(x)\leq \e^{2-\frac{n}{2p}}d_0$, where $\alpha(n,p)=\frac{1}{n+1}(1-\frac{n}{2p}).$

 ii) if A1 and A3 are satisfied, then for any $q>0$ and $\lambda> n-2q$, there exist a small constant $\overline{\e}=\overline{\e}(n,\delta)$, and a constant $C=C(n,q,\lambda,\k,\Gamma,\delta)$ such that for any $0<\e\leq\overline{\e}$, we have
\begin{equation}\label{general excess estimate1}
\oint_{B_{\e d_0}(x)}e(y)dy\leq C\left[e(x)+\e^2d_0+(||\psi^+||_{K^{q,\lambda}}+||\psi^-||_{K^{q,\lambda}})\e^{2-\frac{n-\lambda}{q}}d_0^{2-\frac{n-\lambda}{q}}\right],
\end{equation}
for all $x\in M_{\frac{\delta}{4},16}$.

\end{lemma}

\noindent {\bf Proof.}
Inequalities \eqref{general excess estimate} and \eqref{general excess estimate1} follows directly from Lemma \ref{parabolicL1harnack} applied to time independent functions, since
$$\Delta e(x)=\Delta d^--\Delta d^+\leq \frac{C}{d_0}+\psi^{-}+\psi^{+}.$$ To see that \eqref{general excess estimate} implies \eqref{excess estimate}, notice that for some $q>1$ satisfying $2\e^q\leq \e$, and any $y\in B_{(\e-\e^q)d_0}(x)$, we have
$$
\int_{B_{\e^qd_0}(y)}e(z)dz\leq \int_{B_{\e d_0}(x)}e(z)dz\leq C(e(x)+\e^{2-\frac{n}{2p}}d_0)\v(B_{\e d_0}(x))\leq C\e^{2-\frac{n}{2p}}d_0(\e d_0)^n.
$$
Thus,
\begin{equation}
\oint_{B_{\e^qd_0}(y)}e(z)dz\leq C\e^{2-\frac{n}{2p}}d_0 \frac{(\e d_0)^n}{(\e^{q}d_0)^n}=C\e^{2-\frac{n}{2p}+n-nq}d_0.
\end{equation}
This means that there exists a point $y^{\prime}\in B_{\e^q}(y)$ such that
$$e(y^{\prime})\leq C\e^{2-\frac{n}{2p}+n-nq}d_0.$$
Hence,
$$e(y)\leq e(y^{\prime})+2d(y,y^{\prime})\leq C(\e^{2-\frac{n}{2p}+n-nq}+\e^q)d_0.$$
By choosing $q=1+\alpha(n,p)=1+\frac{1}{n+1}(1-\frac{n}{2p})$, one has
$$e(y)\leq C\e^{1+\alpha(n,p)}d_0,$$
for any $y\in B_{\frac{1}{2}\e d_0}(x)\subset B_{(\e-\e^q)d_0}(x)$.\qed\\

Under the assumptions \textit{A1} and \textit{A2}, according to Lemma \ref{cutoffannulus}, we can construct a cut-off function $\phi\geq0$ such that
\begin{equation}\label{cutoff}
\phi=1\ \textrm{on}\ M_{\frac{\delta}{4},8},\ supp(\phi)\subset M_{\frac{\delta}{16},16},\ \textrm{and}\ |\Delta \phi|+|\nabla \phi|^2\leq \frac{C}{d_0^2}.
\end{equation}

Define $h_0^{\pm}(x)=\phi d^{\pm}(x)$, and $e_0(x)=\phi e(x)$. Also, denote by $h_t^{\pm}(x)$ and $e_t(x)=h_t^--h_t^+$ the solutions of the equations
\begin{equation}\label{defh}
\left\{\begin{aligned}
&(\frac{\partial}{\partial t}-\Delta)h^{\pm}(x,t)=0\\
&h^{\pm}(x,0)=h_0^{\pm}(x),
\end{aligned}\right.
\end{equation}
and
\begin{equation}\label{defe}
\left\{\begin{aligned}
&(\frac{\partial}{\partial t}-\Delta)e(x,t)=0\\
&e(x,0)=e_0(x)
\end{aligned}\right.
\end{equation}

In the case where $|Ric^-|\in K^{2,n-2}$, by assuming \textit{A1}, \textit{A3}, and \textit{A4}, and using the method in \cite{CoNa}, one can also show the existence of a cut-off function as in Lemma \ref{cutoffannulus}, and hence construct  $h^{\pm}_t(x)$ and $e_t(x)$ as above.

In the following, we derive estimates of $h^{\pm}_t(x)$ and $e_t(x)$. We will use the notation
$$||\psi^{\pm}||_{K^{q,\lambda}}:=||\psi^-||_{K^{q,\lambda}}+||\psi^+||_{K^{q,\lambda}}.$$

Following \cite{CoNa}, we first have
\begin{lemma}\label{laplacian upper bound}
i) If A1 and A2 are satisfied, then there exists a constant $C=C(n,p,\k,\Lambda,\delta)$ such that
$$\Delta h_t^-, -\Delta h_t^+, \Delta e_t\leq C\left(\frac{1}{d_0}+(||\psi^+||_{L^{2p}}+||\psi^-||_{L^{2p}})t^{-\frac{n}{4p}}\right)\leq C\left(\frac{1}{d_0}+t^{-\frac{n}{4p}}\right)$$
in $M_{\frac{\delta}{16}, 16}$.\\

ii) If A1, A3, and A4 are satisfied, then for any $q>0$ and $\lambda>0$, there exists a constant  $C=C(n,q,\lambda, \k,\Gamma,\delta)$ such that
$$\Delta h_t^-, -\Delta h_t^+, \Delta e_t\leq C\left(\frac{1}{d_0}+||\psi^{\pm}||_{K^{q,\lambda}} t^{\frac{n-\lambda}{2q}}\right).$$
in $M_{\frac{\delta}{16}, 16}$.
\end{lemma}

Part i) above is Lemma 2.20 in \cite{TZz:1}, and the proof of part ii) is similar.


Using the Li-Yau bound in section 2 and Bochner's formula, it can be shown that
\begin{proposition}\label{prop excess}
i) If A1 and A2 are satisfied, there exists a constant $C=C(n,p,\k,\Lambda,\delta)$, such that for any $x\in M_{\frac{\delta}{2},4}$ and $0<t\leq \overline{\e}^2d_0^2$, the following estimates hold for $y\in B_{\sqrt{t}}(x)$,\\
(1) $\displaystyle |e_t(y)|\leq C\left(e(x)+td_0^{-1}+t^{1-\frac{n}{4p}}\right)$. \\
(2)$\displaystyle |\nabla e_t|(y)\leq \frac{C}{\sqrt{t}}\left(e(x)+td_0^{-1}+t^{1-\frac{n}{4p}}\right)$.\\
(3) $\displaystyle \left|\frac{\partial}{\partial t}e_t(y)\right|=|\Delta e_t(y)|\leq \frac{C}{t}\left(e(x)+td_0^{-1}+t^{1-\frac{n}{4p}}\right)$.\\
(4) $\displaystyle \oint_{B_{\sqrt{t}}(y)}|\nabla^2 e_t|^2\leq \frac{C}{t^2}\left(e(x)+td^{-1}_0+t^{1-\frac{n}{4p}}\right)^2$.\\

ii) If If A1, A3, and A4 are satisfied, then for any $q>0$ and $\lambda>n-2q$, there exists a constant $C=C(n,q,\lambda,\Gamma,\delta)$ such that\\
(1') $\displaystyle |e_t(y)|\leq C\left(e(x)+td_0^{-1}+||\psi^{\pm}||_{K^{q,\lambda}} t^{1-\frac{n-\lambda}{2q}}\right)$. \\
(2')$\displaystyle |\nabla e_t|(y)\leq \frac{C}{\sqrt{t}}\left(e(x)+td_0^{-1}+||\psi^{\pm}||_{K^{q,\lambda}} t^{1-\frac{n-\lambda}{2q}}\right)$.\\
(3') $\displaystyle \left|\frac{\partial}{\partial t}e_t(y)\right|=|\Delta e_t(y)|\leq \frac{C}{t}\left(e(x)+td_0^{-1}+||\psi^{\pm}||_{K^{q,\lambda}} t^{1-\frac{n-\lambda}{2q}}\right)$.\\
(4') $\displaystyle \oint_{B_{\sqrt{t}}(y)}|\nabla^2 e_t|^2\leq \frac{C}{t^2}\left(e(x)+td^{-1}_0+||\psi^{\pm}||_{K^{q,\lambda}} t^{1-\frac{n-\lambda}{2q}}\right)^2$.\\
Here $\overline{\e}$ is the constant in Lemma \ref{excess}.
\end{proposition}

The estimates (1) and (2) above are included in Lemma 2.21 in \cite{TZz:1}. The proofs of (3) and (4) are similar to Lemma 2.11 in \cite{CoNa}, which rely on the Li-Yau bound in Theorem \ref{thmLY}. The proof of part ii) is similar to part i).

Alternatively, one may also use the Gaussian estimate of $|\frac{\partial G}{\partial t}|$ to obtain the estimates in the above proposition.

From the lemma above, one gets

\begin{lemma}\label{c0approximation}
i) If A1 and A2 are satisfied, then for any $x\in M_{\frac{\delta}{2}, 4}$, we have
$$|h_t^{\pm}(x)-d^{\pm}(x)|\leq C(e(x)+td_0^{-1}+t^{1-\frac{n}{4p}}).$$
ii) If A1, A3, and A4 are satisfied, then for any $q>0$, $\lambda>n-2q$ and $x\in M_{\frac{\delta}{2}, 4}$, we have
$$|h_t^{\pm}(x)-d^{\pm}(x)|\leq C(e(x)+td_0^{-1}+||\psi^{\pm}||_{K^{q,\lambda}} t^{1-\frac{n-\lambda}{2q}}).$$
\end{lemma}

Part i) above is (2.41) in \cite{TZz:1}, and the proof of part ii) follows similarly. Recall from \cite{CoNa} that an $\e$-geodesic connecting $O^+$ and $O^-$ is a unit speed curve $\sigma$ such that $\left||\sigma|-d_0\right|\leq \e^2d_0$. Moreover, one has

\begin{lemma}{(Colding-Naber \cite{CoNa})} \label{egeodesic}\\
1) Let $\sigma$ be an $\e$-geodesic connecting $O^+$ and $O^-$. Then for any $z\in \sigma$, we have $e(z)\leq \e^2d_0$.\\
2) Let $x\in M$ such that $e(x)\leq \e^2d_0$. Then there exists an $\e$-geodesic $\sigma$ such that $x\in \sigma$.
\end{lemma}

From Lemma \ref{c0approximation} and Lemma \ref{egeodesic}, we immediately have

\begin{corollary}\label{c0egeodesic}
For any $\e$-geodesic $\sigma$ connecting $O^+$ and $O^-$, any $x\in \sigma\bigcap M_{\delta/2, 4}$, and $0<\e\leq\overline{\e}$, we have\\
i) when A1 and A2 are satisfied, $$\left|h^{\pm}_{d^2_{\e}}-d^{\pm}\right|\leq C(\e^2d_0+\e^{2-\frac{n}{2p}}d_0^{2-\frac{n}{2p}}).$$
ii) when A1, A3, and A4 are satisfied, for $q>0$ and $\lambda>n-2q$ we have
$$\left|h^{\pm}_{d^2_{\e}}-d^{\pm}\right|\leq C(\e^2d_0+||\psi^{\pm}||_{K^{q,\lambda}} \e^{2-\frac{n-\lambda}{q}}d_0^{2-\frac{n-\lambda}{q}}).$$
Here $\overline{\e}$ is the constant in Lemma \ref{excess}.
\end{corollary}

To prove that $h^{\pm}_t$ are $L^1$ close to $d^{\pm}$, we first need the following lemma.

\begin{lemma}
For any $x\in M_{\frac{\delta}{2}, 4}$, we have
$$\int_{M_{\frac{\delta}{16}, 16}\setminus M_{\frac{\delta}{4},8}} G(y,t;x,0)dy\leq \frac{C}{d_0^2}t.$$
\end{lemma}

\noindent {\bf Proof.}
Let $k$ be a positive integer, and $\phi_k\geq0$ a cut-off function such that $\phi_k=1$ in $M_{\frac{\delta}{k},2k}$, $\phi_k=0$ outside of $M_{\frac{\delta}{2k},4k}$, and $|\Delta \phi_k|+|\nabla \phi_k|^2\leq \frac{C}{d_0^2}$.
Then, for any $x$, we have
\begin{align*}
\left|\frac{d}{dt}\int_\M \phi_k(y)G(y,t;x,0)dy\right|&=\left|\int_\M \phi_k(y)\Delta_yG(y,t;x,0)dy\right|\\
&=\left|\int_\M \Delta\phi_k(y) G(y,t;x,0)dy\right|\\
&\leq \frac{C}{d_0^2}.
\end{align*}
Thus, we have
$$-\frac{C}{d_0^2}t\leq -\phi_k(x)+\int_{M} \phi_k(y)G(y,t;x,0)dy\leq \frac{C}{d_0^2}t.$$
It follows that for $x\in M_{\frac{\delta}{2}, 4}$,
$$\int_{M_{\frac{\delta}{16}, 32}}G(y,t;x,0)dy\leq 1+\frac{C}{d_0^2}t,$$
and
$$\int_{M_{\frac{\delta}{4}, 8}}G(y,t;x,0)dy\geq 1-\frac{C}{d_0^2}t.$$
Therefore,
$$\int_{M_{\frac{\delta}{16}, 16}\setminus M_{\frac{\delta}{4},8}} G(y,t;x,0)dy\leq \frac{C}{d_0^2}t.$$
\qed

By using the above lemma and following the proof of Lemma 2.17 in \cite{CoNa}, we can get

\begin{lemma}\label{c1approximation}
For any $x\in M_{\frac{\delta}{2}, 4}$, and $0<t<\overline{\e}^2d_0^2$\\

i) if A1 and A2 are satisfied, then we have
\begin{align*}
|\nabla h^{\pm}_t|^2(x)\leq 1+\frac{C}{d_0^2}t+Ct^{1-\frac{n}{2p}}.
\end{align*}

ii) if A1, A3, and A4 are satisfied, then
\begin{align*}
|\nabla h^{\pm}_t|^2(x)\leq 1+\frac{C}{d_0^2}t+C\sqrt{t}.
\end{align*}
Here $\overline{\e}$ is the constant in Lemma \ref{excess}.
\end{lemma}

The above $C^1$ bound of $h^{\pm}_t$ can be applied to show that

\begin{lemma}\label{L1approximation}

i) If A1 and A2 are satisfied, then there exists a constant $C=C(n,p,\k,\Lambda,\delta)$, such that for any $0<\e\leq\overline{\e}$ and $0<\sqrt{t}<\e^2d_0^2$, we have\\
(1) If $x\in M_{\delta, 2}$, and $e(x)\leq \e^2d_0$, then $\displaystyle \oint_{B_{10\sqrt{t}}(x)}\left||\nabla h^{\pm}_{t}|^2(y)-1\right|dy\leq Ct^{-\frac{1}{2}}(\e^2d_0+td_0^{-1}+t^{1-\frac{n}{4p}})$.\\
(2) If $\sigma$ is an $\e$-geodesic connecting $O^+$ and $O^-$, then
$$\oint_{\delta d_0}^{(1-\delta)d_0}\oint_{B_{10\sqrt{t}}(\sigma(s))}\left||\nabla h^{\pm}_{t}|^2(y)-1\right|dyds\leq C(\e^2d_0+td_0^{-1}+t^{1-\frac{n}{2p}}d_0^{\frac{n}{2p}}).$$

ii) If A1, A3, and A4 are satisfied, then for any $q>0$ and $\lambda>n-2q$, there exists a constant $C=C(n,q,\lambda,\k,\Gamma,\delta)$, such that for any $0<\e\leq\overline{\e}$ and $0<\sqrt{t}<\e^2d_0^2$, we have\\
(1') If $x\in M_{\delta, 2}$, and $e(x)\leq \e^2d_0$, then $$\displaystyle \oint_{B_{10\sqrt{t}}(x)}\left||\nabla h^{\pm}_{t}|^2(y)-1\right|dy\leq Ct^{-\frac{1}{2}}(\e^2d_0+td_0^{-1}+||\psi^{\pm}||_{K^{q,\lambda}}t^{1-\frac{n-\lambda}{2q}}).$$
(2') If $\sigma$ is an $\e$-geodesic connecting $O^+$ and $O^-$, then
$$\oint_{\delta d_0}^{(1-\delta)d_0}\oint_{B_{10\sqrt{t}}(\sigma(s))}\left||\nabla h^{\pm}_{t}|^2(y)-1\right|dyds\leq C(\e^2d_0+td_0^{-1}+\sqrt{t}d_0+||\psi^{\pm}||_{K^{q,\lambda}}t^{1-\frac{n-\lambda}{2q}}).$$
Here $\overline{\e}$ is the constant in Lemma \ref{excess}.
\end{lemma}

The proof of the above Lemma is similar to Lemma 2.18 in \cite{CoNa}. Now we are ready to give a\\

\hspace{-.5cm}{\bf Proof of Theorem \ref{distance approximation}:}
Estimates (1), (2), and (3) are contained in Lemmas \ref{c0approximation} and \ref{L1approximation}, respectively. In the following, we prove (4).

For any $\sigma(s)$, let $\eta(x)\geq0$ be the cut-off function satisfying $\eta=1$ in $B_{d_{\e}}(\sigma(s))$,$\eta=0$ outside of $B_{3d_{\e}}(\sigma(s))$, and $|\Delta\eta|+|\nabla\eta|^2\leq \frac{C}{d_{\e}^2}$, where $d_{\e}=\e d_0$.

Let $a(t)$ be a smooth function in time such that $0\leq a(t)\leq 1$, $a(t)=1$ for $t\in[\frac{1}{2}d_{\e}^2, 2d_{\e}^2]$, $a(t)=0$ for $t\not\in[\frac{1}{4}d^2_{\e}, 4d^2_{\e}]$, and $|a^{\prime}(t)|\leq\frac{C}{d^2_{\e}}$.

Recall that $$(\frac{\partial}{\partial t}-\Delta)\left(|\nabla h^{\pm}_t|^2-1\right)=-2|\nabla^2 h^{\pm}_{t}|^2-2R_{ij}\nabla_i h^{\pm}_t\nabla_j h^{\pm}_t.$$
Hence, we have
\begin{align*}
&2\int_\M  a(t)\eta|\nabla^2 h^{\pm}_t|^2\\
=&\int_\M  a(t)\eta\Delta(|\nabla h^{\pm}_t|^2-1)-2\int_\M  a(t)\eta R_{ij}\nabla_i h^{\pm}_t\nabla_j h^{\pm}_t-\int_\M  a(t)\eta\frac{\partial}{\partial t}(|\nabla h^{\pm}_t|^2-1)\\
\leq & \frac{C}{d_{\e}^2}\int_{B_{3d_{\e}}(\sigma(s))}a(t)||\nabla h^{\pm}_t|^2-1|+2\int_{B_{3d_{\e}}(\sigma(s))}a(t)|Ric^-||\nabla h^{\pm}_t|^2-\int_\M  a(t)\eta\frac{\partial}{\partial t}(|\nabla h^{\pm}_t|^2-1).
\end{align*}
Therefore,
\begin{align*}
&\frac{1}{\v(B_{3d_{\e}}(\sigma(s)))}\int_{\frac{1}{2}d_{\e}^2}^{2d_{\e}^2}\int_{B_{d_{\e}}(\sigma(s))}|\nabla^2 h^{\pm}_t|^2 dydt\\
\leq & \frac{C}{d_{\e}^2}\int_{\frac{1}{4}d_{\e}^2}^{4d_{\e}^2}\oint_{B_{3d_{\e}}(\sigma(s))}||\nabla h^{\pm}_t|^2-1|dydt+C\int_{\frac{1}{4}d_{\e}^2}^{4d_{\e}^2}\oint_{B_{3d_{\e}}(\sigma(s))}|Ric^-|(1+\frac{C}{d_0^2}t+Ct^{1-\frac{n}{2p}})dydt\\
\end{align*}
It follows immediately from \eqref{noncollapsing} and Theorem \ref{volumecomparison} that
\begin{align*}
&\int_{\frac{1}{2}d_{\e}^2}^{2d_{\e}^2}\int_{\delta d_0}^{(1-\delta)d_0}\oint_{B_{d_{\e}}(\sigma(s))}|\nabla^2 h^{\pm}_t|^2 dydsdt\\
\leq & \frac{C}{d_{\e}^2}\int_{\frac{1}{4}d_{\e}^2}^{4d_{\e}^2}(\e^2d_0+\frac{t}{d_0}+t^{1-\frac{n}{2p}}d_0^{\frac{n}{2p}})dt+Cd_0d_{\e}^{-\frac{n}{p}}\int_{\frac{1}{4}d_{\e}^2}^{4d_{\e}^2}(1+\frac{C}{d_0^2}t+Ct^{1-\frac{n}{2p}})dt\\
\leq& Cd_0(\e^2+\e^{2-\frac{n}{p}}).
\end{align*}

Therefore, there exists a $r\in[\frac{1}{2}, 2]$ such that
\begin{equation}
\int_{\delta d_0}^{(1-\delta)d_0}\oint_{B_{d_{\e}}(\sigma(s))}|\nabla^2 h^{\pm}_{rd_{\e}^2}|^2 dyds\leq \frac{C(1+\e^{-\frac{n}{p}})}{d_0}.
\end{equation} This completes the proof of Theorem \ref{distance approximation}.
\qed\\

The proof of Theorem \ref{distance approximation 2} is done similarly.


\section*{Acknowledgements}
The paper was originally intended to be a joint one with Richard Bamler who made an important contribution.
However, Richard has withdrawn from authorship to concentrate on other projects. We sincerely thank him for his generosity.

M. Z. would like to express his great appreciation to Prof. Huai-Dong Cao for constant encouragement and support. We are grateful to Professors H.-D. Cao, X.Z. Dai, H.Z. Li,
G.F. Wei for their interest and comments on the result.

Q.S.Z. gratefully acknowledges the support of Simons' Foundation and Siyuan Foundation through Nanjing University. Meng Zhu's research is supported by National Natural Science Foundation of China Grant 11501206, and Science and Technology Commission of Shanghai Municipality (Grant No. 13dz2260400).

Note added on 2017/09/06.  After this paper was posted on the arxiv in November 2015,  C. Rose treated the
case that $|Ric|$ is in the Kato class instead of $|Ric|^2$ in Kato class as in Theorem 1.1 (b) here (see \cite{Ro}).

Finally, we are grateful to the referee for carefully checking the paper and making useful suggestions.

\end{document}